\documentclass[reqno,12pt]{amsart}
\usepackage{latexsym}
\usepackage{amsfonts}
\usepackage{amssymb}
\usepackage{mathrsfs}
\usepackage[all]{xy}
\usepackage{bbm}
\usepackage{color}
\usepackage[mathscr]{euscript}
\usepackage{tikz}
\usepackage{amsmath,amsthm,amssymb,mathrsfs,amsfonts,verbatim,color,leftidx}
\usepackage{amsbsy}
\usepackage[colorlinks,urlcolor=blue,linkcolor=blue,citecolor=brown]{hyperref}
\usetikzlibrary{matrix,arrows}

\newtheorem{theorem}{Theorem}[section]
\newtheorem{corollary}[theorem]{Corollary}
\newtheorem{lemma}[theorem]{Lemma}

\newtheorem{question}{Question}
\newtheorem{proposition}[theorem]{Proposition}
\theoremstyle{definition}
\newtheorem{example}{Example}
\newtheorem{definition}[theorem]{Definition}
\newtheorem{remark}[theorem]{Remark}
\numberwithin{equation}{section}

\newcommand{\bTheta}{\boldsymbol{\Theta}}

\newcommand{\RR}{\mathbb{R}}

\newcommand{\NN}{\mathbb{N}}

\newcommand{\QQ}{\mathbb{Q}}
\newcommand{\ZZ}{\mathbb{Z}}

\newcommand{\cS}{\mathcal{S}}

\newcommand{\cA}{\mathcal{A}}

\newcommand{\cM}{\mathcal{M}}

\newcommand{\cH}{\mathcal{H}}

\newcommand{\ba}{\mathbf{a}}
\newcommand{\bb}{\mathbf{b}}

\newcommand{\bh}{\mathbf{h}}

\newcommand{\bv}{\mathbf{v}}

\newcommand{\bM}{\mathbf{M}}

\newcommand{\bp}{\mathbf{p}}

\newcommand{\bL}{\mathbf{L}}

\newcommand{\bq}{\mathbf{q}}

\newcommand{\dist}{\mathrm{dist}}

\newcommand{\SL}{\mathrm{SL}}

\newcommand{\dd}{\; \mathrm{d}}

\newcommand{\bone}{\mathbbm{1}}

\newcommand{\eps}{\varepsilon}

\newcommand{\Sing}{\mathbf{Sing}}

\newcommand {\ignore}[1]  {}

\newif\ifdraft\drafttrue
\draftfalse

\newcommand{\ggm}{G/\Gamma}

\newcommand{\btheta}{{\boldsymbol{\theta}}}

\newcommand{\sm}{\smallsetminus}

\renewcommand{\setminus}{\smallsetminus}

\begin{document}

%\title{Singular points on product of certain homogeneous spaces}
\title{Divergent trajectories on products of homogeneous spaces}
\author{Jinpeng An}
\address{School of Mathematical Sciences, Peking University, Beijing, 100871, China}
\email{anjinpeng@gmail.com}

\author{Lifan Guan}
\address{Mathematisches Institut, Georg-August Universit{\"a}t G{\"o}ttingen, Bunsenstrasse 3-5, D-37073 Gottingen, Germany}
\email{guanlifan@gmail.com}

\author{Antoine Marnat}
\address{Universit\"at Wien, Vienna, Austria}
\email{antoine.marnat@univie.ac.at}

\author{Ronggang Shi}
\address{Shanghai Center for Mathematical Sciences, Jiangwan Campus, Fudan University, No.2005 Songhu Road, Shanghai, 200438, China}
\email{ronggang@fudan.edu.cn}

\thanks{J. An is supported by NSFC grant 11322101, L. Guan is supported by ERC Consolidator grant 648329, A. Marnat is supported by FWF Project I 3466-N35, and R. Shi is supported by NSFC grant 11871158.}

\maketitle

\begin{abstract}
%In this paper, we extend the dimension formula for singular matrices to products of certain homogeneous spaces.
In this paper, we determine the Hausdorff dimension of the set of points with divergent trajectories on the product of certain homogeneous spaces. The flow is allowed to be weighted with respect to the factors in the product space. The result is derived from its counterpart in Diophantine approximation. In doing this, we introduce a notion of jointly singular matrix tuples, and extend the dimension formula for singular matrices to such matrix tuples.
\end{abstract}

\section{Introduction}

The roots of the theory of Diophantine approximation lie in Dirichlet's Theorem. It asserts that for any real matrix $\btheta \in M_{m\times n}(\RR)$, the system of inequalities
$$\| \btheta \bq - \bp \|^m < Q^{-1} \quad \text{and} \quad   0< \|\bq\|^n \le Q $$
has an  integer solution $(\bp, \bq) \in \ZZ^m\times  \ZZ^n$ for any real number $Q\ge1$. Here $\|\cdot\|$ denotes the supremum  norm.  One of the central topics in Diophantine approximation is  to investigate matrices for which one can go beyond Dirichlet's Theorem. In this regard, a crucial object of study is the class of singular matrices introduced by Khintchine \cite{KhinSing1,KhinSing2}. A comprehensive survey on this topic is Moshchevitin \cite{MoshSurv}.

Recall that a matrix $\btheta \in M_{m\times n}(\RR)$ is \emph{singular} if for every $\epsilon>0$, the system of  inequalities
\begin{equation}\label{Sing}
\| \btheta \bq - \bp \|^m < \epsilon Q^{-1}  \quad \text{and} \quad   0< \|\bq\|^n \le Q
\end{equation}
has an integer solution $(\bp, \bq) \in \ZZ^m\times  \ZZ^n$ for any sufficiently large real number $Q$. Let $\Sing_{m,n}$ denote the set of all singular matrices in $M_{m\times n}(\RR)$. It is well-known that $\Sing_{1,1}=\QQ$. In general, a classical result of Khintchine states that $\Sing_{m,n}$ has Lebesgue measure $0$. Regarding the Hausdorff dimension, breakthroughs have been made recently by several groups of authors, which together give the following formula.

\begin{theorem}\cite{Ch,ChCh,KKLM,VarPrinc}\label{t-dfsu}
For any $(m,n)\in \NN^2$ with $(m, n)\ne (1, 1)$, we have
$$\dim \Sing_{m,n} = mn-\frac{mn}{m+n}.$$
\end{theorem}

Here and throughout the paper, ``$\dim$" refers to the Hausdorff dimension. The $(m,n)=(2,1)$ case of Theorem \ref{t-dfsu} is due to Cheung \cite{Ch}, and the $n=1$ case is due to Cheung and Chevallier \cite{ChCh}. In the general case, the sharp upper bound of $\dim \Sing_{m,n}$
was obtained by Kadyrov, Kleinbock, Lindenstrauss and Margulis in \cite{KKLM} using the contraction property of the height function, and the sharp  lower bound was obtained by Das, Fishman, Simmons and Urba\'nski \cite{VarPrinc} using a variational principle for parametric geometry of numbers. A related question is to calculate the dimension of weighted singular linear forms. The  dimension of   weighted singular vectors in $\mathbb R^2$ was obtained by Liao, Solan, Tamam and the forth named author in \cite{LSST}.

Thanks to Dani's correspondence \cite{Dani}, many Diophantine properties of $\btheta$ can be reformulated dynamically. Let $m,n \in \NN$, and let $$Y_{m+n}=\SL(m+n, \RR)/\SL(m+n, \ZZ).$$
Consider the one-parameter semigroup
\begin{equation}\label{eq;f+}
  F_{m,n}^+=\{g_t^{(m,n)}: t\ge 0\}, \quad\text{ where }\quad g_t^{(m,n)}=\left(\begin{array}{cc}e^{t/m}I_m &   \\   & e^{-t/n}I_n\end{array}\right).\footnote{Here the time parametrization of $g_t^{(m,n)}$ is chosen to be compatible with the definition of a template in \cite{VarPrinc}. It differs from that in \cite{KKLM} by a factor $mn$.}
\end{equation}
For $\btheta\in M_{m\times n}(\RR)$, denote
\begin{equation}\label{eq;utheta}
u_{\btheta}=\begin{pmatrix}
                I_m & \btheta \\
                0 & I_n
              \end{pmatrix} \quad\text{ and }\quad x_{\btheta}=u_{\btheta}\ZZ^{m+n}\in Y_{m+n}.
 \end{equation}
Then $\btheta\in M_{m\times n}(\RR)$ is singular if and only if the trajectory $F_{m,n}^+ x_{\btheta}$ is \emph{divergent}, i.e., it eventually leaves every compact subset of $Y_{m+n}$.

In general, let $G$ be a noncompact Lie group, $\Gamma\subset G$ be a nonuniform lattice, and $F^+=\{g_t: t\ge 0\}\subset G$ be a one-parameter subsemigroup. Let us say that a point $x\in \ggm$ is \emph{$F^+$-singular} if the corresponding trajectory $F^+x$ is divergent on $\ggm$.
The set $D(F^+, \ggm)$ of {$F^+$-singular} points has been extensively studied in recent years.  A related notion  was introduced in \cite{KKLM}:
For $\delta\in (0,1]$, let us say that a point $x\in \ggm$ is
\emph{$(F^+,\delta)$-singular}\footnote{Such points are called \emph{$\delta$-escape on average} in \cite{KKLM}. Here we use the terminology ``singular points" to emphasize their relation to singular matrices.}
if for any compact subset $K$ of $\ggm$, one has
\begin{equation*}
\limsup_{T\to \infty}\frac{1}{T}\int_0^T \mathbbm{1}_K (g_t x)\dd t\le 1-\delta,
\end{equation*}
where $\mathbbm{1}_K$ denotes the characteristic function of $K$. The set of $(F^+,\delta)$-singular points is denoted by $D_{\delta}(F^+, \ggm)$. As a dynamical counterpart of Theorem \ref{t-dfsu}, we have the following natural question.

\begin{question}\label{q-basic-dynamics}
What are $\dim D(F^+, \ggm)$ and $\dim D_{\delta}(F^+, \ggm)$?
\end{question}

In the most general case, it is proved in \cite{GS} that $\dim D_1(F^+, \ggm)<\dim \ggm$.
As a direct corollary of Theorem \ref{t-dfsu}, we have %for any $(m,n)\in \NN^2$ with $(m,n)\ne (1,1)$ that
\begin{equation}\label{e-dim-dmn}
 \dim D(F_{m,n}^+, Y_{m+n})=\dim Y_{m+n}-\frac{mn}{m+n}, \qquad (m,n)\in \NN^2\sm\{(1,1)\}.
\end{equation}
The sharp upper and lower bounds of $\dim D_{\delta}(F_{m,n}^+, Y_{m+n})$ were also obtained in \cite{KKLM}  and \cite{VarPrinc}, respectively, which together give:

\begin{theorem}[\cite{KKLM,VarPrinc}]\label{T:1.4}
Let $(m,n)\in \NN^2$ and $\delta\in (0,1]$. Then
\begin{equation}\label{e-dfsu-2}
  \dim D_{\delta}(F_{m,n}^+, Y_{m+n})= \dim Y_{m+n}-\delta\frac{mn}{m+n}.
\end{equation}
\end{theorem}

In this paper, we consider certain special cases of Question \ref{q-basic-dynamics}, namely, when the system $(F^+, \ggm)$ is a product of homogeneous systems. More precisely, let $s\ge 2$ be an integer, and let
\begin{equation}\label{eq;notation}
G=\prod_{i=1}^s G_i, \quad \Gamma=\prod_{i=1}^s \Gamma_i, \quad X_i=G_i/\Gamma_i, \quad X=\ggm=\prod_{i=1}^s X_i,
\end{equation}
where $G_i=\SL(m_i+n_i, \RR)$, $\Gamma_i=\SL(m_i+n_i, \ZZ)$, and $(m_i, n_i)\in \NN^2$. Let
$$A^+=\prod_{i=1}^s F_i^+,$$
where $F_i^+=F_{m_i, n_i}^+$ is given in  (\ref{eq;f+}).
Let $F^+$ be a one-parameter subsemigroup of $ A^+$
that projects non-trivially to each component.
The homogeneous system $(F^+, \ggm)$ is the main object of our study.
For any such   $F^+$, there exists $\ba=(a_1, \ldots, a_s)\in \RR_+^s$, where $\RR_+=(0,\infty)$, such that
\begin{equation}\label{eq;flow}
F^+=F^+_{\ba}:=\left\{g_t=\left(g_{a_1t}^{(m_1,n_1)},\ldots, g_{a_st}^{(m_s,n_s)} \right): t\ge 0\right\}.
\end{equation}
We say that $F^+_{\ba}$ is the one-parameter subsemigroup of $A^+$ associated to the  \emph{weight vector} $\ba$.
Note that $F^+_{\ba}=F^+_{\ba'}$ if and only if $\ba=c\,\ba'$ for some positive constant $c$.

For $1\le j\le s$, consider the natural projections $G\to\prod_{i\ne j}G_i$ and $X\to\prod_{i\ne j}X_i$. By abuse of notation, we denote both projections by $\pi_j$. Note that for $x\in X$, if some $\pi_j(x)$ is $\pi_j(F^+_\ba)$-singular (resp.~$(\pi_j(F^+_\ba), \delta)$-singular), then $x$ is $F^+_\ba$-singular (resp. $(F^+_\ba, \delta)$-singular). This motivates us to make the following definition.

\begin{definition}
A point $x\in X$ is \emph{essentially $F^+_\ba$-singular} (resp.~\emph{essentially $(F^+_\ba, \delta)$-singular}) if it is $F^+_\ba$-singular (resp.  $(F^+_\ba, \delta)$-singular) but for each $1\le j\le s$, $\pi_j(x)$ is not $\pi_j(F^+_\ba)$-singular (resp. not $(\pi_j(F^+_\ba),\delta)$-singular).
\end{definition}

It is easy to see that if $x\in X$ is essentially $F^+_\ba$-singular, then the divergent trajectory $F^+_\ba x$ is non-obvious in the sense of \cite{W}, hence is non-degenerate in the sense of \cite{Dani}.

Let us denote the set of essentially  $F^+_\ba$-singular (resp.~essentially $(F^+_\ba, \delta)$-singular) points by $D^e(F^+_\ba, X)$ (resp. $D^e_{\delta}(F^+_\ba, X))$. The main result  of this paper is  as follows.

\begin{theorem}\label{t-main}
Let $X$ be the product of $s$ homogeneous spaces given by (\ref{eq;notation})  with $s\ge 2$, and let $F_\ba^+$ be the one-parameter semigroup in (\ref{eq;flow}) associated to the weight vector $\ba\in \RR^s_+$.  Then
\begin{equation}\label{e-main1}
\dim D(F^+_\ba, X)=  \dim D^e(F^+_\ba, X)=\dim X- \min_{1\le i\le s} \frac{m_in_i}{m_i+n_i}
\end{equation}
and, for any $\delta\in (0, 1]$,
\begin{equation}\label{e-main2}
 \dim D_{\delta}(F^+_\ba, X)= \dim D^e_{\delta}(F^+_\ba, X)=\dim X-\delta\min_{1\le i\le s} \frac{m_in_i}{m_i+n_i}.
\end{equation}
\end{theorem}

%Note that $1$-singularity does not imply  singularity (see Remark \ref{r-compare-sing}).  Hence
%\eqref{e-main2} does not imply  \eqref{e-main1}.
%%It follows from the definition that
%%\begin{equation*}
%%D(F^+_\ba, X)=D^e(F^+_\ba, X) \cup\bigcup_{i=1}^s \left\{(x_1,\ldots, x_s)\in X: x_i \text{ is $F^+_i$-singular} %\right\}
%%\end{equation*}
%%and
%%\begin{equation*}
%%D_{\delta}(F^+_\ba, X)=D^e_{\delta}(F^+_\ba, X) \cup\bigcup_{i=1}^s \left\{(x_1,\ldots, x_s)\in X: x_i \text{ is %$(F^+_i, \delta)$-singular} \right\}.
%%\end{equation*}
%%In view of \eqref{e-dim-dmn} and \eqref{e-dfsu-2}, the right hand sides of (\ref{e-main1}) and (\ref{e-main2})
%%are lower bounds of $\dim D(F^+_\ba, X)$ and $ \dim D_{\delta}(F^+_\ba, X)$, respectively.
%\red{Obviously, $D^e(F^+_\ba, X) \subset D(F^+_\ba, X)\subset D_\delta(F^+_\ba, X)$ and $D^e_{\delta}(F^+_\ba, X)\subset D_{\delta}(F^+_\ba, X)$ for any $\delta\in (0,1].$ }
%So the main points of Theorem \ref{t-main} are the sharp upper bound of $\dim D_{\delta}(F^+_\ba, X)$ and
%the sharp lower bounds of $\dim D^e(F^+_\ba, X)$ and $ \dim D^e_{\delta}(F^+_\ba, X)$.

Theorem \ref{t-main} is new even for $X=\big(\SL(2,\RR)/\SL(2,\ZZ)\big)^s$, namely, when all the $(m_i, n_i)$ are equal to $(1, 1)$.
In this case, the dimension formula for $D(F^+_\ba, X)$ was conjectured by Y. Cheung via private communication with the fourth named author.
Cheung's motivation is his result in \cite{Cheung}  where he proved the formula in the case where $(m_i,n_i)=(1, 1)$ and $\ba=(1, \ldots, 1)$.
An extension of Cheung's result to products of hyperbolic spaces can be found in \cite{Y}. Let us also remark that if all the $(m_i, n_i)$ are different from $(1,1)$, the sharp upper bound of $\dim D(F^+_\ba, X)$, together with \eqref{e-dim-dmn}, is enough to imply the dimension formula for $D(F^+_\ba, X)$. However, if some $(m_i, n_i)$ is $(1,1)$, we need to estimate the dimension from both sides.

Note that the right hand sides of (\ref{e-main1}) and (\ref{e-main2}) are independent of the weight vector $\ba$. In fact,
our method of proof also implies that the Hausdorff dimensions of $\bigcup_{\ba \in \RR_+^s}D(F^+_\ba, X)$ and
$\bigcup_{\ba \in \RR_+^s}D_\delta (F^+_\ba, X)$
are equal to the right hand sides of (\ref{e-main1}) and (\ref{e-main2}), respectively. We will explain this at the end of Section \ref{sec;set}. Let us also remark that  the dimension formulas in Theorem \ref{t-main}
are local. This means that for any non-empty open subset $U$ of $X$, the intersections of the various singular points sets with $U$
have the same dimensions as themselves.

Theorem \ref{t-main} will be derived from its counterpart in Diophantine approximation, namely Proposition \ref{p-main} in the next section. Roughly speaking, Proposition \ref{p-main} gives the Hausdorff dimensions of certain sets of matrix tuples that are ``jointly singular". In particular, it shows that if $s\ge2$, then the $s$-tuples $(\theta_1,\ldots,\theta_s)\in\RR^s$ such that for every $\epsilon>0$ and every sufficiently large $Q$, there exists $q\in\NN$ with
\begin{equation}\label{joint0}
\min_{1\le i\le s}\dist(q\theta_i,\ZZ)<\epsilon Q^{-1} \quad \text{and} \quad q \le Q
\end{equation}
form a set of Hausdorff dimension $s-1/2$. Note that if the ``$\min$" sign in \eqref{joint0} is replaced by ``$\max$" and the term $\epsilon Q^{-1}$ is replaced by $\epsilon Q^{-1/s}$, we get the definition of $s$-dimensional singular vectors, which form a set of Hausdorff dimension $s-\frac{s}{s+1}$ by \cite{ChCh}. The precise definition of joint singularity and more examples will be given in Section \ref{s-reduction}.
\\

The organization of this paper is as follows. In Section \ref{s-reduction},  we deduce Theorem \ref{t-main} from its Diophantine counterpart Proposition \ref{p-main}, which gives the dimension formulas of   singular  points sets on an unstable horospherical leaf.
The proof of Proposition \ref{p-main} is the main  body of the paper and is given in two independent sections.

In Section \ref{UB}, we estimate the dimension from above using the covering theorem in \cite{KKLM}.
When the dynamical system $(F_\ba^+, X)$ has only one positive Lyapunov exponent,
the optimal  upper bound follows rather directly from the arguments in \cite{KKLM}. Otherwise, essential new ideas are needed, see Remark \ref{R:dfct}. We construct  a universal covering of the set of singular points
independent of the weight $\ba$. The key step is Lemma \ref{l-key}, which forms the main innovative part of Section \ref{UB}.
This method can also be used to give the sharp upper bound of the dimension of singular points set in other product systems.

In Section \ref{s-lower}, we give the estimates from below   using the variational principle in parametric geometry of numbers introduced in \cite{VarPrinc}. The variational principle enables us to  study a very large family of Diophantine sets, namely, that can be described using \emph{templates}, which are certain piecewise linear functions. In particular, the Hausdorff dimension of a Diophantine set that is associated to a certain template can be computed using only the information of  the template. By carefully choosing templates for all the $(m_i, n_i)$, we manage to construct certain product sets of matrices that give the optimal lower bounds.
\\

{\bf Acknowledgments:} We would like to thank Yitwah Cheung for introducing his conjecture and Weixiao Shen for the discussions on obtaining the optimal upper bound.

\section{Joint singularity of matrix tuples}\label{s-reduction}

In this section, we deduce Theorem \ref{t-main} from its counterpart in Diophantine approximation, which concerns ``joint singularity properties" of matrix tuples.  Let us fix an integer $s\ge 2$, a pair $(m_i, n_i)\in \NN^2$ for each $1\le i\le s$, and denote
\[M_i=M_{m_i\times n_i}(\RR)\quad \text{ and }\quad \bM=\prod_{i=1}^s M_i.\]
For $\bTheta=(\btheta_1,\ldots, \btheta_s)\in \bM$, let
\[u_{\bTheta}=(u_{\btheta_1},\ldots, u_{\btheta_s})\in G \quad\text{ and }\quad x_{\bTheta}=(x_{\btheta_1},\ldots, x_{\btheta_s})\in X, \]
where  $u_\btheta$ and $x_\btheta$ are as in (\ref{eq;utheta}).
It is easily checked that $U:=\{u_{\bTheta}: \bTheta\in \bM\}$ is the expanding horospherical subgroup of $G$ with respect to $F^+_\ba$ for any $\ba\in \RR^s_+$.
Consider the following sets of matrix tuples:
\begin{align}
D(F^+_\ba, \bM)&=\{\bTheta\in \bM: x_{\bTheta} \text{ is  } F_{\ba}^+\text{-singular}\},\label{e:D1}\\
D^e(F^+_\ba, \bM)&=\{\bTheta\in \bM: x_{\bTheta} \text{ is essentially } F_{\ba}^+\text{-singular}\},\label{e:D2}\\
D_{\delta}(F^+_\ba, \bM)&=\{\bTheta\in \bM: x_{\bTheta} \text{ is } (F_{\ba}^+, \delta)\text{-singular}\},\label{e:D3}\\
D^e_{\delta}(F^+_\ba, \bM)&=\{\bTheta\in \bM: x_{\bTheta} \text{ is essentially } (F_{\ba}^+, \delta)\text{-singular}\}.\label{e:D4}
\end{align}
In order to give the Diophantine interpretations of these sets, let us introduce some notation. For $\btheta \in M_{m\times n}(\RR)$ and $\epsilon>0$, let $\mathcal{Q}_\epsilon(\btheta)$ denote the set of all real numbers $Q\ge1$ such that the system of  inequalities \eqref{Sing} has an integer solution $(\bp, \bq) \in \ZZ^m\times  \ZZ^n$.
Moreover, for $\ba=(a_1, \ldots, a_s)\in \RR_+^s$, $\bTheta=(\btheta_1,\ldots, \btheta_s)\in \bM$, $\epsilon>0$ and $1\le j\le s$, denote
\begin{align*}
\mathcal{Q}_{\ba,\epsilon}(\bTheta)&:=\bigcup_{i=1}^s\mathcal{Q}_\epsilon(\btheta_i)^{1/a_i},\\
\mathcal{Q}_{\ba,\epsilon,j}(\bTheta)&:=\bigcup_{i\ne j}\mathcal{Q}_\epsilon(\btheta_i)^{1/a_i}.
\end{align*}
Let us say that a subset of $\RR$ is a \emph{neighborhood of $+\infty$} if it contains the interval $(C,+\infty)$ for some $C\in\RR$. Clearly, $\btheta$ is singular if and only if $\mathcal{Q}_\epsilon(\btheta)$ is a neighborhood of $+\infty$ for every $\epsilon>0$. The first statement in the following lemma generalizes this to matrix tuples.

\begin{lemma}\label{L:joint}
Let $\ba\in\RR_+^s$, $\bTheta\in \bM$, and  $\delta\in (0,1]$.
\begin{enumerate}
  \item $\bTheta\in D(F^+_\ba, \bM)$ if and only if $\mathcal{Q}_{\ba,\epsilon}(\bTheta)$ is a neighborhood of $+\infty$ for every $\epsilon>0$.
  \item $\bTheta\in D^e(F^+_\ba, \bM)$ if and only if $\mathcal{Q}_{\ba,\epsilon}(\bTheta)$ is a neighborhood of $+\infty$ for every $\epsilon>0$ but $\mathcal{Q}_{\ba,\epsilon_0,j}(\bTheta)$ is not a neighborhood of $+\infty$ for some $\epsilon_0>0$ and every $1\le j\le s$.
  \item $\bTheta\in D_{\delta}(F^+_\ba, \bM)$ if and only if
\begin{equation}\label{E:union}
\liminf_{T\to\infty}\frac{1}{T}\int_0^T \mathbbm{1}_{\mathcal{Q}_{\ba,\epsilon}(\bTheta)} (e^t)\dd t\ge\delta
\end{equation}
  for every $\epsilon>0$. %, where $\mu$ is the Lebesgue measure.
  \item $\bTheta\in D^e_{\delta}(F^+_\ba, \bM)$ if and only if \eqref{E:union} holds for every $\epsilon>0$ but
  $$\liminf_{T\to\infty}\frac{1}{T}\int_0^T \mathbbm{1}_{\mathcal{Q}_{\ba,\epsilon_0,j}(\bTheta)} (e^t)\dd t<\delta$$
  for some $\epsilon_0>0$ and every $1\le j\le s$.
\end{enumerate}
\end{lemma}

\begin{proof}
We only prove (1). The proofs of (2)--(4) are in the same spirit and are left to the reader.

Suppose $\ba=(a_1, \ldots, a_s)$. From Mahler's compactness criterion, it is easy to see that $\bTheta\in D(F^+_\ba, \bM)$ if and only if the following statement holds:
\begin{itemize}
  \item[$(*)$] For every $c\in(0,1]$, there exists $t_0\ge0$ such that for every $t>t_0$, there exists $1\le i\le s$ and $(\bp,\bq)\in\ZZ^{m_i}\times\ZZ^{n_i}$ such that
$\|\btheta_i\bq-\bp\|<ce^{-a_it/m_i}$ and $0<\|\bq\|<ce^{a_it/n_i}$.
\end{itemize}
Suppose statement $(*)$ holds. Let $\epsilon>0$, we prove that $\mathcal{Q}_{\ba,\epsilon}(\bTheta)$ is a neighborhood of $+\infty$. Without loss of generality, we may assume $\epsilon\le1$. Applying statement $(*)$ with $c=\epsilon$, we see that there exists $t_0\ge0$ such that for every $t> t_0$, there exist $1\le i\le s$ and $(\bp,\bq)\in\ZZ^{m_i}\times\ZZ^{n_i}$ such that
$$\|\btheta_i\bq-\bp\|^{m_i}<\epsilon^{m_i}e^{-a_it}\le\epsilon e^{-a_it}$$
and
$$0<\|\bq\|^{n_i}<\epsilon^{n_i} e^{a_it}\le e^{a_it}.$$
Thus $\mathcal{Q}_{\ba,\epsilon}(\bTheta)$ contains $(e^{t_0},+\infty)$, hence is a neighborhood of $+\infty$.

Conversely, suppose that for every $\epsilon>0$, $\mathcal{Q}_{\ba,\epsilon}(\bTheta)$ is a neighborhood of $+\infty$.
To prove statement $(*)$, let $c\in(0,1]$. Let $C_0\ge1$ be such that $(C_0,+\infty)\subset\mathcal{Q}_{\ba,\epsilon}(\bTheta)$ with
$$\epsilon=\min_{1\le j\le s}c^{m_j+n_j},$$
and let
$$t_0=\log(2C_0)-(\log c)\max_{1\le j\le s}n_j/a_j.$$
Then, for $t>t_0$, the number $\min_{1\le j\le s}c^{n_j/a_j}e^t/2$ is greater than $C_0$, hence is in $\mathcal{Q}_\epsilon(\btheta_i)^{1/a_i}$ for some $1\le i\le s$. This means that there exists $(\bp,\bq)\in\ZZ^{m_i}\times\ZZ^{n_i}$ such that
$$\|\btheta_i\bq-\bp\|<\epsilon^{1/m_i} \left(\min_{1\le j\le s}c^{n_j/a_j}e^t/2\right)^{-a_i/m_i}<ce^{-a_it/m_i}$$
and
$$0<\|\bq\|\le \left(\min_{1\le j\le s}c^{n_j/a_j}e^t/2\right)^{a_i/n_i}<ce^{a_it/n_i}.$$
Thus statement $(*)$ holds. This proves (1).
\end{proof}

Lemma \ref{L:joint} tells us that the sets \eqref{e:D1}--\eqref{e:D4} consist of matrix tuples that are ``jointly singular" in certain senses, with the weight vector $\ba$. In particular, we say that a matrix tuple $(\btheta_1,\ldots, \btheta_s)$ is \emph{jointly $\ba$-singular} if it is in the set $D(F^+_\ba, \bM)$. Let us explain some special cases more explicitly.

\begin{example}
Suppose $(m_1,n_1)=\cdots=(m_s,n_s)=(m,n)$ and $\ba=(1,\ldots,1)$. Then $(\btheta_1,\ldots, \btheta_s)\in M_{m\times n}(\RR)^s$ is jointly $\ba$-singular if and only if for every $\epsilon>0$, the system of  inequalities
\begin{equation}\label{joint}
\min_{1\le i\le s}\| \btheta_i \bq - \bp \|^m < \epsilon Q^{-1} \quad \text{and} \quad 0< \|\bq\|^n \le Q
\end{equation}
has an integer solution $(\bp, \bq) \in \ZZ^m\times  \ZZ^n$ for any sufficiently  large real number $Q$. Note that if $(m,n)=(1,1)$, then \eqref{joint} is equivalent to \eqref{joint0}. Proposition \ref{p-main} below shows that the set of such tuples has Hausdorff dimension $mn(s-\frac{1}{m+n})$. \qed
\end{example}

\begin{example}
Suppose $s=2$ and $(m_1,n_1)=(m_2,n_2)=(1,1)$. Then for $(a_1,a_2)\in\RR_+^2$, a pair $(\theta_1,\theta_2)\in\RR^2$ is jointly $(a_1,a_2)$-singular  if and only if for every $\epsilon>0$ and every sufficiently large $C$, there exists $q\in\NN$ such that either
\begin{equation}\label{e:theta1}
\dist(q\theta_1,\ZZ)<\epsilon C^{-a_1} \quad \text{and} \quad q\le C^{a_1},
\end{equation}
or
\begin{equation}\label{e:theta2}
\dist(q\theta_2,\ZZ)<\epsilon C^{-a_2} \quad \text{and} \quad q\le C^{a_2}.
\end{equation}
If we instead require that one of \eqref{e:theta1} and \eqref{e:theta2} is always satisfied, then $\theta_1$ or $\theta_2$ is rational. However, Proposition \ref{p-main} implies that the set of jointly $(a_1,a_2)$-singular pairs has Hausdorff dimension $3/2$. \qed
\end{example}

%Then Theorem \ref{t-main} can be deduced from the following proposition.

Let us now formulate the Diophantine counterpart of Theorem \ref{t-main}.

\begin{proposition}\label{p-main}
Let the notation be as above. Then
$$ \dim D(F^+_\ba, \bM) = \dim D^e(F^+_\ba, \bM)=\sum_{i=1}^s m_in_i - \min_{1\le i\le s} \frac{m_in_i}{m_i+n_i},$$
and
 $$\dim D_{\delta}(F^+_\ba, \bM)= \dim D^e_{\delta}(F^+_\ba, \bM)=\sum_{i=1}^s m_in_i-\delta\min_{1\le i\le s} \frac{m_in_i}{m_i+n_i}.$$
\end{proposition}

The proof of Proposition \ref{p-main} will occupy the next two sections. In the rest of this section, we derive Theorem \ref{t-main} from Proposition \ref{p-main}.

\begin{proof}[Proof of Theorem \ref{t-main} modulo Proposition \ref{p-main}]
Let $P$ be the weakly contracting subgroup of $G$ with respect to $F^+_\ba$, i.e.,
$$P=\left\{h\in G: \text{ the set } \{ghg^{-1}: g\in F^+_\ba\} \text{ is bounded}\right\}.$$
	Then $P$ is a parabolic subgroup of $G$ whose Lie algebra is complementary to the Lie algebra of $U$. It is straightforward to verify that the set $PU:=\{pu: p\in P, u\in U \}$ consists of elements $(g_1,\ldots,g_s)$ in $G$ such that for each $1\le i\le s$, the submatrix of $g_i$ formed by its first $m_i$ rows and first $m_i$ columns is invertible. In particular, $PU$ is Zariski open in $G$.
On the other hand, by Borel's density theorem \cite{B}, every left coset of $\Gamma$ is Zariski dense in $G$.
It follows that the map
$$\pi: P\times \bM\rightarrow X, \quad (p, \bTheta)\mapsto px_{\bTheta}$$
is surjective.

Note that for any $p\in P$ and $x\in X$, if  $px$ is (essentially) $F_{\ba}^+$-singular or (essentially) $(F_{\ba}^+, \delta)$-singular, then so is    $x$.
Hence we have
$$\pi^{-1}( D(F^+_\ba, X))=P\times  D(F^+_\ba, \bM),$$
and  similar  equalities  with $D$  replaced by $D^e, D_\delta $ or $D^e_\delta$.
Since the multiplication map $P\times U\to PU $ is a diffeomorphism (see, e.g., \cite[Lemma 6.44]{Knapp}), the map  $\pi$ is a local diffeomorphism. Thus we have
\[\dim D(F^+_\ba, X)= \dim \pi^{-1}(D(F^+_\ba, X)),\]
and similar  equalities with $D$  replaced by $D^e, D_\delta $ or $D^e_\delta$.
Note that for any subset $Y$ of $\bM$, $\dim (P\times Y)= \dim P+\dim Y$.
So the dimension formulas   in
Theorem \ref{t-main} follow from Proposition \ref{p-main} and the fact that $\dim \bM=\sum_{i=1}^s m_in_i$.
\end{proof}

\section{The upper bounds}\label{UB}

The aim of this section is to  estimate    the dimensions in Proposition \ref{p-main} from above. Clearly, we have
$$D_{\delta}^e(F^+_\ba, \bM)\subset D_{\delta}(F^+_\ba, \bM)  \quad \text{and} \quad D^e(F^+_\ba, \bM)
\subset D(F^+_\ba, \bM) \subset  D_{1}(F^+_\ba, \bM).$$
So the sharp upper bounds of their dimensions  will follow from the following proposition.

\begin{proposition}\label{p-upper-bound}
Let $\delta\in (0,1]$ and $\ba\in \RR_+^s$, then
\begin{equation*}
 \dim D_{\delta}(F^+_\ba, \bM)\le \sum_{i=1}^s m_in_i-\delta\min_{1\le i\le s} \frac{m_in_i}{m_i+n_i}.
\end{equation*}
\end{proposition}

\subsection{Auxiliary sets}\label{sec;set}
In this section we cover $D_{\delta}(F^+_\ba, \bM)$ by sets whose dimensions are easier to estimate from above.

For any $1\le i \le s$, we choose and fix a right invariant Riemannian metric $\dist_i(\cdot, \cdot)$ on $G_i$, which naturally induces a metric on $X_i=G_i/\Gamma_i$, also denoted by $``\dist_i"$, as follows:
\begin{equation*}
  \dist_i(g\Gamma_i, h\Gamma_i)=\inf_{\gamma\in \Gamma_i} \dist(g\gamma, h), \text{ where } g, h\in G_i.
\end{equation*}
Set $``\dist"$ to be the metric on $X$ given by
 $$\dist((x_1,\ldots, x_s),(y_1,\ldots, y_s))=\max_{1\le i\le s} \dist_i(x_i,y_i).$$
 For $R>0$, let
\begin{equation*}
B_{R}^{X}= \{x\in X: \dist(x, [1_G])\le R\} \quad \mbox{and} \quad E_{R}^{X}=X\setminus B_R^X,
\end{equation*}
  where $[1_G]$ denotes the coset of the   identity element $1_G$.
For $R, T>0$ and $0<\delta \le 1$,  let
\begin{equation}\label{e-def-set1}
\widetilde D_{\delta}(F_{\ba}^+, R, T)=\left\{\bTheta\in \bM: \frac{1}{T} \int_{0}^T \bone_{E_{R}^{X}}(g_t x_{\bTheta}) \dd t\ge \delta \right\}.
\end{equation}
The value $\frac{1}{T} \int_{0}^T \bone_{E_{R}^{X}}(g_t x_{\bTheta}) \dd t$ measures the proportion of the time up to $T$ that the trajectory $F_{\ba}^+ x_{\bTheta}$ spends in the set $E_R^{X}$. Thus, the set $\widetilde D_{\delta}(F_\ba ^+, R, T)$ can be thought of as  an approximation to the set $D_{\delta}(F_{\ba}^+, \bM)$.
Their precise relation can be stated as follows:
for any $ 0< \delta'<\delta\le 1$ and $R>0$, we have
\begin{equation}\label{e-set-app}
D_{\delta}(F_{\ba}^+, \bM)\subset \liminf_{T\rightarrow \infty}\widetilde  D_{\delta'}(F_{\ba}^+, R, T):= \bigcup_{T_1>0}\bigcap_{T>T_1}\widetilde D_{\delta'}(F_{\ba}^+, R, T).
\end{equation}
This gives our   first enlargement of $D_{\delta}(F_{\ba}^+, \bM)$.

Next we cover each $\widetilde D_{\delta'}(F_{\ba}^+, R, T)$ by a set
 defined using the data on each component of $X=\prod_{i=1}^s X_i$.
For $1\le i\le s$ and $R>0$, we set
\begin{equation*}
  B_{R}^{X_i}=\{x\in X_i: \dist_i(x, [1_{G_i}])\le R\}\quad \mbox{and}\quad
   E_{R}^{X_i}=X_i\setminus B_{R}^{X_i} .
  \end{equation*}
We write  $g_{i, t}= g_{ t}^{(m_i, n_i)} $ to simplify the notation.
 For $R, T>0$ and $\btheta\in M_i$, set
\begin{equation*}
\cA_i( R, T, \btheta)=\frac{1}{T} \int_{0}^T \bone_{E_{R}^{X_i}}(g_{i,t} x_\btheta) \dd t.
\end{equation*}
Since $\dist$ is defined as the maximum of all the  $\dist_i$, we have
\begin{equation*}
\frac{1}{T} \int_{0}^T \bone_{E_{R}^{X}}(g_{t} x_{\bTheta}) \dd t=
\frac{1}{T} \int_{0}^T \max_{1\le i\le s}\bone_{E_{R}^{X_i}}(g_{i,a_it} x_{\btheta_i}) \dd t
\le \cA( F_\ba ^+, R, T, \bTheta),
\end{equation*}
where
\begin{equation*}
\cA( F_\ba^+, R, T, \bTheta)=\sum_{i=1}^s \cA_i( R,a_i T, \btheta_i).
\end{equation*}
This together with (\ref{e-def-set1}) implies
\begin{equation}\label{e-def-set2}
  \widetilde D_{\delta'}(F_{\ba}^+, R, T)\subset D_{\delta'}(F_{\ba}^+, R, T):=\left\{\bTheta\in \bM: \cA(F_\ba^+, R, T, \bTheta)\ge \delta'\right\}.
\end{equation}
Combining \eqref{e-set-app} and \eqref{e-def-set2}, we get, for any $0< \delta'< \delta\le 1$,
$$D_{\delta}(F_{\ba}^+, \bM)\subset \liminf_{T\rightarrow \infty}  D_{\delta'}(F_{\ba}^+, R, T):= \bigcup_{T_1>0}\bigcap_{T>T_1}D_{\delta'}(F_{\ba}^+, R, T).$$
We summarize what we have obtained in the following lemma.
\begin{lemma}\label{lem;sum}
Suppose  $0< \delta'< \delta\le 1$, then 	
\begin{align}\label{eq;repeat}
D_{\delta}(F_{\ba}^+, \bM)\subset \liminf_{T\rightarrow \infty}  D_{\delta'}(F_{\ba}^+, R, T).
\end{align}
\end{lemma}

The key step in our proof of Proposition \ref{p-upper-bound} is that the right hand side of (\ref{eq;repeat})
is contained in the   limsup set associated to  any weight vector.  More precisely, we have
the following lemma.

\begin{lemma}\label{l-inc}
	Let $0<\delta'<\delta\le 1$ and $\ba, \mathbf b\in \RR_+^s$, then
	\begin{equation}\label{eq;limsup}
 \liminf_{T\rightarrow \infty}  D_{\delta}(F_{\ba}^+, R,  T)\subset \limsup_{T\rightarrow \infty} D_{\delta'}(F_{\mathbf b}^+, R, T):=\bigcap_{T_1>0}\bigcup_{T>T_1}D_{\delta'}(F_{\mathbf b}^+, R, T).
	\end{equation}
\end{lemma}
The proof of Lemma \ref{l-inc} will be given in Section \ref{sec;key}. The above two lemmas reduce the proof  of Proposition \ref{p-upper-bound} to estimating the dimension of the right hand side of (\ref{eq;limsup}) for a convenient weight $\mathbf b$.
The special weight we are using will be
\begin{equation}\label{eq;bprime}
{\mathbf b}_0=\left(\frac{m_1n_1}{m_1+n_1},\ldots,\frac{m_sn_s}{m_s+n_s}\right).
\end{equation}
In this case the dynamical system $(F^+_{{\mathbf b}_0}, X)$ has a single positive Lyapunov exponent.
By Lemmas \ref{lem;sum} and \ref{l-inc}, for any $0<\delta'<\delta\le 1$,
\begin{align}\label{eq;clear}
D_\delta(F^+_\ba, \bM )\subset D_{\delta'} :=\bigcap_{R>0}\bigcap_{T_1>0}\bigcup_{T>T_1}D_{\delta'}(F_{{\mathbf b}_0}^+, R, T).
\end{align}
So Proposition \ref{p-upper-bound} will follow from the following lemma.
\begin{lemma}\label{l-special}
Let $\delta\in (0,1]$, then
\begin{equation}\label{eq;typo}
\dim D_\delta
\le \sum_{i=1}^s m_in_i-\delta\min_{1\le i\le s} \frac{m_in_i}{m_i+n_i}.
\end{equation}
\end{lemma}
The proof  of Lemma \ref{l-special} will be given in Section \ref{sec;spcial}.
Since the right hand side of (\ref{eq;clear}) does not depend on $\ba\in \RR^s_+$, Lemma
\ref{l-special} also implies
\begin{equation*}
\dim \left (\bigcup_{\ba\in \RR_+^s} D_{\delta}(F^+_\ba, \bM)\right)\le \sum_{i=1}^s m_in_i-\delta\min_{1\le i\le s} \frac{m_in_i}{m_i+n_i}.
\end{equation*}

\subsection{Proof of Lemma \ref{l-inc}}\label{sec;key}
The proof of Lemma \ref{l-inc} is based on the the following key lemma.
\begin{lemma}\label{l-key}
Let $s\in \NN$, $1=\sigma_1\ge \sigma_2\ge \ldots \sigma_s>0$ and $f_1, f_2, \ldots f_s: \RR_+ \rightarrow [0, \infty)$ be bounded functions. Then for any $\epsilon>0$ and $t_0>0$, there exists $t\ge t_0$ such that
\begin{equation}\label{e-ineq}
  \sum_{i=1}^s f_i(t)\le \epsilon +\sum_{i=1}^s f_i(\sigma_i t).
\end{equation}
\end{lemma}

\begin{proof}
	
  We argue by induction on $s$. For $s=1$, since $\sigma_1=1$, the inequality \eqref{e-ineq} is trivial. Suppose $s\ge 2$ and the lemma holds for $s-1$. Let $\epsilon>0$ and $t_0>0$. By assumption, the function $f_s$ is bounded, hence there exists $Q\in \NN$ such that $f_s(x)\le Q\epsilon$ for all $x\in \RR_+$.

  Next  we consider the bounded functions $g_1,\ldots , g_{s-1}: \RR_+ \rightarrow [0, \infty)$ defined as
  \begin{equation}\label{eq;plugin}
    g_i(t)=\sum_{q=0}^Q f_i(\sigma_s^{-q}t),  \quad 1\le i\le s-1.
  \end{equation}
  By the induction hypothesis, there exists $t_1\ge t_0$ such that
  \begin{equation}\label{eq;induction}
    \sum_{i=1}^{s-1}g_i(t_1)\le \epsilon+  \sum_{i=1}^{s-1}g_i(\sigma_i t_1).
  \end{equation}

We claim that
\begin{align}\label{eq;key}
\sum_{q=0}^Q \left(\epsilon+\sum_{i=1}^s f_i(\sigma_i \sigma_s^{-q}t_1)- \sum_{i=1}^s f_i( \sigma_s^{-q}t_1)\right)\ge 0.
\end{align}
 Summing
the index $q$ first and using (\ref{eq;plugin}) for $1\le i\le s-1$, we have  the left hand side of (\ref{eq;key}) is equal to
  \begin{align}
  \notag
     &(Q+1)\epsilon+  \sum_{i=1}^{s-1} g_i(\sigma_i t_1) + \sum_{q=0}^{Q} f_s(\sigma_s^{-q+1} t_1)- \sum_{i=1}^{s-1} g_i( t_1)-\sum_{q=0}^{Q} f_s(\sigma_s^{-q} t_1) \\
     =&\Big(Q\epsilon+f_s(\sigma_s t_1)-f_s(\sigma_s^{-Q} t_1)\Big)+\Big (\epsilon+\sum_{i=1}^{s-1} g_i(\sigma_i t_1)-\sum_{i=1}^{s-1} g_i( t_1)\Big).
     \label{eq;more}
  \end{align}
  The first term of (\ref{eq;more}) is nonnegative since $0\le f_s(x)\le Q\epsilon$
  for all $x\in \RR_+$. The second term of (\ref{eq;more}) is nonnegative by (\ref{eq;induction}). Therefore, (\ref{eq;key}) holds.

By (\ref{eq;key}),  there exists $0\le q\le Q$ such that
  \begin{equation*}
    \epsilon+\sum_{i=1}^s f_i(\sigma_i \sigma_s^{-q}t_1)- \sum_{i=1}^s f_i( \sigma_s^{-q}t_1)\ge 0.
  \end{equation*}
  This implies that $t=\sigma^{-q}_st_1$ satisfies \eqref{e-ineq}.
  Note that $t\ge t_1$, since $\sigma_s\le 1$.
   This  completes the proof.
\end{proof}

\begin{proof}[Proof of Lemma \ref{l-inc}]
Assume the contrary that,  there exists
 $$\bTheta\in \bigcup_{T_1>0}\bigcap_{T>T_1}D_{\delta}(F_{\mathbf a}^+, R, T)\setminus \bigcap_{T_1>0}\bigcup_{T>T_1}D_{\delta'}(F_{\mathbf b}^+, R, T).$$
 Then there exists $ T_1>0$ such that, for any $T\ge T_1$,
\begin{equation*}
\bTheta \in D_{\delta}(F_{\ba}^+, R, T) \quad \text{ but }\quad  \bTheta \notin  D _{\delta'}(F_{\mathbf b}^+, R, T).
\end{equation*}
In view of the definition of  $D_{\delta}(F_{\ba}^+, R, T)$ in (\ref{e-def-set2}), for any $T\ge T_1$,
\begin{equation}\label{e-bound}
\cA(F_{\ba}^+, R, T, \bTheta)\ge \delta\quad  \text{ and }\quad  \cA(F_{\mathbf b}^+, R, T, \bTheta)< \delta'.
\end{equation}

Note that the right hand side of (\ref{eq;limsup}) is unchanged if we rescale
$\mathbf b$.
So by possibly rescaling $ \mathbf b=(b_1, \ldots, b_s)$ and  reordering  $1\le i\le s$ if necessary, we may assume that
\begin{equation*}
  1=\frac{b_1}{a_1}\ge \frac{b_2}{a_2}\ge \cdots \ge \frac{b_s}{a_s}>0.
\end{equation*}
Applying Lemma \ref{l-key} to the functions $f_i(t)=\cA_i(R, a_it,  \btheta_i)$, with
$$\epsilon=\frac{1}{2}(\delta-\delta'), \quad t_0=T_1 \quad \text{ and }  \quad\sigma_i=\frac{b_i}{a_i},$$
we know that  there exists $T\ge T_1$, such that
\begin{align*}
  \cA(F_{\ba}^+, R, T, \bTheta) &= \sum_{i=1}^s\cA_i( R, a_iT, \btheta_i)\\
      &\le \frac{1}{2}(\delta-\delta')+\sum_{i=1}^s\cA_i( R, b_iT, \btheta_i)\\
      &= \frac{1}{2}(\delta-\delta')+ \cA(F_{\mathbf b}^+, R, T, \bTheta).
\end{align*}
This leads to a contradiction to \eqref{e-bound}, hence completes the proof.
\end{proof}

\subsection{Proof of Lemma \ref{l-special}}\label{sec;spcial}

Let us first fix a pair of integers $(m,n)\in \NN^2$. For $r>0$, let $B_r$ denote the open Euclidean ball\footnote{In this subsection, metric balls in vector spaces are assumed to be open. } in $M_{m\times n}(\RR)$ of radius $r$ centered at $0$. The upper bound parts of Theorems \ref{t-dfsu} and \ref{T:1.4} are derived in \cite{KKLM} from the following covering theorem\footnote{Recall that the time parameter $t$ in this paper differs from that in \cite{KKLM} by a factor $mn$.}, which is the main technical result of \cite{KKLM}.

\begin{theorem}[\cite{KKLM}]\label{t-KKLM}
There exist $t_0>0$ and a function $C:Y_{m+n}\to\RR_+$ such that the following holds:
For any $t\ge t_0$, there exists a compact set $K=K(t)$ in $Y_{m+n}$ such that for any $y\in Y_{m+n}$, $\delta\in(0, 1)$ and $\ell \in \NN$, the set
$$Z_y(K,\ell,t,\delta):=\left\{\btheta\in B_1:\#\{k\in \{1,\ldots, \ell\}: g_{kt}^{(m,n)}u_{\btheta}y\notin K\}\ge \delta\ell  \right\}$$
can be covered by no more than
$C(y)(\frac{t}{mn})^{3\ell}e^{(m+n-\delta)\ell t}$ balls in $M_{m\times n}(\RR)$ of radius $e^{-\frac{m+n}{mn}\ell t}$.
\end{theorem}

To prove Lemma \ref{l-special}, we need the following continuous-time analogue of Theorem \ref{t-KKLM}, which is in fact an easy corollary of Theorem \ref{t-KKLM}. For technical reasons, we include the $\delta=0$ case.

\begin{corollary}\label{c-KKLM}
Let $r>0$. Then there exist $T_0>0$ and a function $\tilde{C}:Y_{m+n}\to\RR_+$ such that the following holds:
For any $T\ge T_0$, there exists a compact set $\tilde{K}=\tilde{K}(T)$ in $Y_{m+n}$ such that for any $y\in Y_{m+n}$, $\delta\in [0, 1)$ and $\ell \in \NN$, the set
$$\tilde{Z}_y(r,\tilde{K},\ell T,\delta):=\left\{\btheta\in B_r:\int_{0}^{\ell T} \bone_{Y_{m+n}\setminus \tilde{K}}\left(g_t^{(m,n)}u_\btheta y\right) \dd t\ge \delta\ell T \right\}$$
can be covered by no more than
$\tilde{C}(y)(\frac{T}{mn})^{3\ell}e^{(m+n-\delta)\ell T}$ balls in $M_{m\times n}(\RR)$ of radius $e^{-\frac{m+n}{mn}\ell T}$.
\end{corollary}

\begin{proof}
Let $\btheta_1,\ldots,\btheta_q\in M_{m\times n}(\RR)$ be such that the $q$ unit balls with centers $\btheta_1,\ldots,\btheta_q$ cover $B_r$, and let $C_0>0$ be such that for any $\rho\in(0,1)$, a unit ball in $M_{m\times n}(\RR)$ can be covered by at most $C_0\rho^{-mn}$ balls of radius $\rho$. We claim that
\begin{align*}
T_0&=\max\{t_0,mn\},\\
\tilde{C}(y)&=\sum_{i=1}^q\max\{C(u_{\btheta_i}y),C_0\}, \qquad y\in Y_{m+n},\\
\tilde{K}(T)&=\bigcup_{t\in[0,T]}g_{-t}^{(m,n)}(K(T)), \qquad T\ge T_0
\end{align*}
satisfy the requirement, where $t_0$, $C(\cdot)$ and $K(\cdot)$ are given as in Theorem \ref{t-KKLM}.
Let $T\ge T_0$, $y\in Y_{m+n}$, $\delta\in [0, 1)$,  $\ell \in \NN$. We need to verify that $\tilde{Z}_y(r,\tilde{K}(T),\ell T,\delta)$ can be covered by at most
$\tilde{C}(y)(\frac{T}{mn})^{3\ell}e^{(m+n-\delta)\ell T}$ balls of radius $e^{-\frac{m+n}{mn}\ell T}$.

(1) Suppose $\delta=0$. Then $\tilde{Z}_y(r,\tilde{K}(T),\ell T,\delta)=B_r$, which can be covered by
$$qC_0(e^{-\frac{m+n}{mn}\ell T})^{-mn}\le \tilde{C}(y)(T/mn)^{3\ell}e^{(m+n)\ell T}$$
balls of radius $e^{-\frac{m+n}{mn}\ell T}$.

(2) Suppose $\delta\in (0, 1)$. The definition of $\tilde{K}(T)$ implies that for $x\in Y_{m+n}$ and $k\in\NN$, we have
$$g_{kT}^{(m,n)}x\in K(T) \quad \Longrightarrow \quad g_t^{(m,n)}x\in \tilde{K}(T) \text{ for all } t\in[(k-1)T,kT].$$
It follows that
$$\tilde{Z}_{y'}(1,\tilde{K}(T),\ell T,\delta)\subset Z_{y'}(K(T),\ell,T,\delta)$$
for any $y'\in Y_{m+n}$. Thus, by the choices of $\btheta_1,\ldots,\btheta_q$, we have
\begin{align*}
\tilde{Z}_y(r,\tilde{K}(T),\ell T,\delta)&\subset\bigcup_{i=1}^q\big(\tilde{Z}_{u_{\btheta_i}y}(1,\tilde{K}(T),\ell T,\delta)+\btheta_i\big)\\
&\subset\bigcup_{i=1}^q\big(Z_{u_{\btheta_i}y}(K(T),\ell, T,\delta)+\btheta_i\big).
\end{align*}
Theorem \ref{t-KKLM} implies that each $Z_{u_{\btheta_i}y}(K(T),\ell, T,\delta)$ can be covered by at most $C(u_{\btheta_i}y)(\frac{T}{mn})^{3\ell}e^{(m+n-\delta)\ell T}$ balls of radius $e^{-\frac{m+n}{mn}\ell T}$. Therefore, $\tilde{Z}_y(r,\tilde{K}(T),\ell T,\delta)$ can be covered by at most
$$\sum_{i=1}^qC(u_{\btheta_i}y)(T/mn)^{3\ell}e^{(m+n-\delta)\ell T}\le\tilde{C}(y)(T/mn)^{3\ell}e^{(m+n-\delta)\ell T}$$ balls of the same radius. This completes the verification.
\end{proof}

Let us now return to the context of Lemma \ref{l-special}. We will deduce from Corollary \ref{c-KKLM} a covering result for product spaces.
For simplicity, we write
\begin{equation}\label{e:bi}
b_i=\frac{m_in_i}{m_i+n_i}, \qquad 1\le i\le s.
\end{equation}
Then $\bb_0=(b_1,\ldots,b_s)$, see \eqref{eq;bprime}.
Without loss of generality, assume that
\begin{equation}\label{e:b1}
b_1=\min_{1\le i\le s} b_i.
\end{equation}
Also, for $1\le i\le s$ and $\delta_i\in[0,1)$, denote
$$D_{\delta_i}(F_i^+, R, T)=\left\{\btheta\in M_i: \cA_i( R, T, \btheta)\ge \delta_i \right\}.$$
We first prove the following simple lemma, which approximates $D_{\delta}(F_{\bb_0}^+, R, T)$
by a finite union of product sets.

\begin{lemma}\label{l-easy}
Let $\delta\in (0,1]$, $\epsilon\in (0,\delta)$. Then there exists a finite subset $\cS=\cS(\epsilon)$ of $[0,1)^s$  satisfying the following conditions:
\begin{itemize}
  \item[(1)] For any $(\delta_1, \ldots, \delta_s)\in \cS$,
  $\sum_{i=1}^s \delta_i= \delta-\epsilon.$
  \item[(2)] For any $R,T>0$,
$$          D_{\delta}(F_{ \bb_0}^+, R, T)\subset \bigcup_{(\delta_i)\in \cS}\prod_{i=1}^s D_{\delta_i}(F_i^+, R,  b_i  T).$$
\end{itemize}
\end{lemma}

\begin{proof}
For $a\in(0,1]$, consider the simplex
$$\Sigma_a:=\Big\{(\delta_1,\ldots, \delta_s)\in [0, a]^s: \sum_{i=1}^s \delta_i=a\Big \}.$$
Moreover, for $\bv =(\delta_1, \ldots, \delta_s)\in \Sigma_{\delta-\epsilon}$, denote
$$N_\bv:=\{ (\delta'_1, \ldots,\delta'_s)\in \Sigma_\delta:\delta_i'>\delta_i \text{ for all } 1\le i\le s  \}.$$
Then $\{N_\bv:\bv\in \Sigma_{\delta-\epsilon}\}$ is an open cover of $\Sigma_\delta$. Since $\Sigma_\delta$ is compact, there is a finite subset $\cS$ of $\Sigma_{\delta-\epsilon}$ such that $\Sigma_\delta=\bigcup_{\bv\in\cS}N_\bv$. Note that for $\bv =(\delta_1, \ldots, \delta_s)\in\cS$ and $(\delta'_1, \ldots,\delta'_s)\in N_\bv$, we have
$$D_{\delta'_i}(F_i^+, R,  b_i  T)\subset D_{\delta_i}(F_i^+, R,  b_i  T), \qquad 1\le i\le s.$$
It follows that
\begin{align*}
D_{\delta}(F_{\bb_0}^+,R,T)&\subset\bigcup_{(\delta'_i)\in \Sigma_\delta}\prod_{i=1}^s D_{\delta'_i}(F_i^+,R,b_iT)\\
&=\bigcup_{\bv\in\cS}\bigcup_{(\delta'_i)\in N_\bv}\prod_{i=1}^s D_{\delta'_i}(F_i^+,R,b_iT)\\
&\subset\bigcup_{(\delta_i)\in\cS}\prod_{i=1}^s D_{\delta_i}(F_i^+,R,b_iT).
\end{align*}
It is clear that $\cS\subset[0,1)^s$. So the proof is completed.
\end{proof}

For $1\le i\le s$, let $\mathrm d_i$ denote the Euclidean metric on $M_i$, and $B_r^{M_i}\subset M_i$ denote the Euclidean ball of radius $r$ centered at $0$. Consider the metric
$$\mathrm d((\btheta_1,\ldots, \btheta_s),(\btheta'_1,\ldots, \btheta'_s))=\max_{1\le i\le s}\mathrm d_i(\btheta_i,\btheta'_i)$$
on $\bM$, and let $B_r^{\bM}=B_r^{M_1}\times\cdots\times B_r^{M_s}$ be the associated metric ball of radius $r$ centered at $0$.
For simplicity, we write the right hand side of \eqref{eq;typo} as $\alpha$, that is,
\begin{align*}
\alpha=\Big(\sum_{i=1}^s m_in_i\Big )-\delta b_1.
\end{align*}
Our covering result for product spaces is as follows.

\begin{lemma}\label{l-cover}
For any $r, \epsilon>0$, there exist $T=T(r, \epsilon)>0$ and $R=R(T)>0$ such that for any $\ell\in \NN$, the set
\begin{equation}\label{e-set-cover}
  D_{\delta}(F_{ \bb_0}^+, R, \ell T)\cap  B_r^{\bM}
\end{equation}
can be covered  by no more than $e^{(\alpha+\epsilon)\ell T}$  balls of radius $e^{-\ell T}$.
\end{lemma}

\begin{proof}
Let $\cS(\epsilon/2b_1)\subset[0,1)^s$ be the finite set given as in Lemma \ref{l-easy}. Then the set \eqref{e-set-cover} is covered by
$$ \bigcup_{(\delta_i)\in \cS(\epsilon/2b_1)}\prod_{i=1}^s \left(D_{\delta_i}(F_i^+, R,  b_i \ell T)\cap  B_r^{M_i}\right).$$
Hence it suffices to study the sets
\begin{equation}\label{e-set-cover2}
D_{\delta_i}(F_i^+, R,  b_i \ell T)\cap  B_r^{M_i},
\end{equation}
which coincide with
$\tilde{Z}_{[1_{G_i}]}(r,B_{R}^{X_i},b_i\ell T,\delta_i)$ in the notation of Corollary \ref{c-KKLM}. By Corollary \ref{c-KKLM}, there exist $T_i, C_i>0$ such that, for any $T\ge T_i$, there exists $R_i(T)>0$ such that for any $\ell \in \NN$ and $R\ge R_i(T)$, the set \eqref{e-set-cover2} can be covered by no more than
$$C_i(b_iT/m_in_i)^{3\ell}e^{(m_i+n_i-\delta_i)b_i\ell T}\le C_iT^{3\ell} e^{(m_in_i-\delta_ib_i)\ell T}$$
balls of radius $e^{-\ell T}$.
Hence for any $T\ge \max_{1\le i\le s}T_i$ and $R \ge \max_{1\le i\le s}R_i(T)$, the set
$\prod_{i=1}^s \left(D_{\delta_i}(F_i^+, R,  b_i \ell T)\cap  B_r^{M_i}\right)$ can be covered by no more  than
\[C T^{3s\ell}e^{\sum_{i=1}^{s} (m_in_i-\delta_ib_i)\ell T}\]
balls of radius $e^{-\ell T}$, where $C=\prod_{i=1}^s C_i$. Taking $T_0$ large enough, we may assume that
$$\#\cS(\epsilon/2b_1)\cdot C T^{3s\ell}\le e^{\frac{\epsilon\ell T}{2}}$$
for any $T\ge T_0$ and $\ell\in \NN$.
On the other hand, by the choice of $\cS(\epsilon/2b_1)$, we have
\[\sum_{i=1}^{s} (m_in_i-\delta_ib_i)\le \sum_{i=1}^{s} m_in_i-\left(\sum_{i=1}^{s} \delta_i\right)b_1=\alpha+\frac{\epsilon}{2}.\]
In summary, for any $T\ge \max_{0\le i\le s} T_i$ and $R\ge \max_{1\le i\le s} R_i(T)$,
the  set \eqref{e-set-cover} can be covered by no more than
\begin{align*}
\# \cS(\epsilon/2b_1)\cdot  \max_{(\delta_i)\in \cS(\epsilon/2b_1)} C T^{3s\ell} e^{\sum_{i=1}^{s} (m_in_i-\delta_ib_i)\ell T}\le e^{(\alpha+\epsilon)\ell T}
\end{align*}
balls of radius  $e^{-\ell T}$. This proves the lemma.
\end{proof}

\begin{remark}\label{R:dfct}
The above argument does not generalize directly to prove a similar covering result for a general weight vector $\ba$. This is the main difficulty in proving Proposition \ref{p-upper-bound} and is resolved by Lemma \ref{l-inc} above.
\end{remark}

We are now prepared to prove Lemma \ref{l-special}.

\begin{proof}[Proof of Lemma \ref{l-special}]
By the definition of Hausdorff dimension, it suffices to show that for any $r>0$ and $ \sigma>\alpha$, the Hausdorff measure
\begin{equation*}
  \cH^{ \sigma}( D _\delta\cap  B_r^{\bM})=0.
\end{equation*}
Recall that, for a subset $Z\subset \bM$,
\begin{equation*}
\cH^{ \sigma}(Z)=\lim_{\beta\rightarrow 0}\cH^{ \sigma}_\beta(Z),
\end{equation*}
where
\begin{equation*}
 \cH^{ \sigma}_\beta(Z)=\inf \left\{ \sum_k |U_k|^{ \sigma}: Z\subset \bigcup_{k}U_k, |U_k|\le \beta \right\}.
\end{equation*}
Hence it suffices to show that for any $\beta>0$,
\begin{equation}\label{e-wanna-prove}
  \cH^{ \sigma}_\beta( D_\delta \cap  B_r^\bM)=0.
\end{equation}

We claim that
 for any $R>0$ and $T>0$
\begin{equation}\label{eq;rig}
   D _\delta \subset \bigcap_{\ell_1\in \NN} \bigcup_{\ell \in \NN, \ell\ge \ell_1}
    D_{\delta} (F_{ \bb_0}^+, R, \ell T).
\end{equation}
According to the definition of $D_\delta$ in  \eqref{eq;clear}, it suffices to prove that
there exists $R_1>R$ such that for any $t\in [(\ell -1)T, \ell T]$, if
$\bTheta\in  D_{\delta} (F_{ \bb_0}^+, R_1, t)$, then
$$\bTheta\in  D_{\delta} (F_{ \bb_0}^+, R, (\ell-1) T)\cup  D_{\delta} (F_{ \bb_0}^+, R, \ell T). $$
If $\bTheta\not \in  D_{\delta} (F_{ \bb_0}^+, R, (\ell-1) T)$, then
$\bTheta\not \in  D_{\delta} (F_{ \bb_0}^+, R_1, (\ell-1) T)$. The assumption
$\bTheta\in  D_{\delta} (F_{ \bb_0}^+, R_1, t)$ implies that there exists
$t_1 \in [(\ell-1)T, \ell T]$ such that $g_{ t_1 }  x_{\bTheta }\in E^X_{R_1}$.
The claim now  follows by taking $R_1 $ sufficiently large  so that,  if
$g_{t_1} x_{\bTheta} \in E_{R_1}^X$
for some $t_1\in  [(\ell -1)T, \ell T]$, then $g_{t_2}x_{\bTheta}\in  E_{R}^X$ for all
$t_2\in [(\ell -1)T, \ell T]$.

By applying Lemma \ref{l-cover} to $\epsilon=\frac{1}{2}( \sigma-\alpha)$ and $r$, we can find
$T>0$ and $R>0 $ such that for any $\ell \in \NN$, the set
\[
D_\delta(F_{\bb_0}^+, R, \ell T)\cap B_r^{\bM}
\]
can be covered by no more than $e^{(\sigma+\alpha)\ell T/2}$ balls of radius $e^{-\ell T}$.
Suppose $\ell_1$ is large enough such that $2e^{-\ell_1 T}\le \beta$.
By (\ref{eq;rig}),
\[
D_\delta \cap B_r^\bM \subset \bigcup_{\ell \in \NN, \ell\ge \ell_1}
D_{\delta} (F_{ \bb_0}^+, R, \ell T)\cap B_r^\bM.
\]
It follows that
\begin{align*}
  \cH^{ \sigma}_\beta(D_{\delta}\cap B_r^\bM)
   &\le \sum_{\ell\ge \ell_1}  \cH^{ \sigma}_{\beta}( D _{\delta}(F_{ \bb_0}^+, R, \ell T)\cap B_r^\bM) \\
   &\le \sum_{\ell \ge \ell_1} e^{(\sigma+\alpha)\ell T/2} e^{- \sigma \ell T}\\
   &=\frac{e^{-\frac{1}{2}( \sigma-\alpha)\ell_1 T}}{1-e^{-\frac{1}{2}( \sigma-\alpha) T}}.
\end{align*}
By letting $\ell_1$ go to infinity, we get  \eqref{e-wanna-prove}.
 This completes the proof.
\end{proof}

\section{The lower bounds}\label{s-lower}

This section is devoted to proving the following lower bounds in Proposition \ref{p-main}.

\begin{proposition}\label{p-lower-bound}
For  $\ba\in \RR_+^s$ and $\delta\in (0,1]$, we have
\begin{equation}\label{e-lower-bound-1}
  \dim D^e(F^+_\ba, \bM )\ge \sum_{i=1}^sm_in_i- \min_{1\le i\le s} \frac{m_in_i}{m_i+n_i},
\end{equation}
and
\begin{equation}\label{e-lower-bound-2}
 \dim D^e_{\delta}(F^+_\ba, \bM )\ge\sum_{i=1}^sm_in_i-\delta\min_{1\le i\le s} \frac{m_in_i}{m_i+n_i}.
\end{equation}
\end{proposition}

The main tool of the proof is the \emph{variational principle} in parametric geometry of numbers developed by Das, Fishman, Simmons and Urba\'nski in \cite{DFSU17,VarPrinc}. It allows us to construct a set of points with a given Diophantine property,
whose Hausdorff dimension is computable. Before heading to the proof of Proposition \ref{p-lower-bound}, we first recall some basics of parametric geometry of numbers.

\subsection{Parametric geometry of numbers and variational principle}

Parametric geometry of numbers originates in a question of Schmidt \cite{SchmidtLuminy}. It was developed by Schmidt and Summerer \cite{SS09,SS13} and Roy \cite{Roy1}. Recently, Das, Fishman, Simmons and Urba\'nski \cite{DFSU17,VarPrinc} established a variational principle, which generalizes and quantifies an important theorem of Roy and has been a powerful tool for computing Hausdorff dimensions.

Let $m,n \in \NN$ and $\btheta \in M_{m\times n} (\RR)$.
The main purpose of parametric geometry of numbers is to study the trajectory $\{g_t^{(m,n)}x_\btheta: t\ge 0\}\subset Y_{m+n}$ through the \emph{successive minima function}
$$\bh = \bh_\btheta := (h_{\btheta,1}, \ldots , h_{\btheta,m+n}) : [0,\infty) \to \RR^{m+n}$$
where for $1\le k\le m+n$ and $t\ge0$,
\[ h_{\btheta,k}(t) = \log \lambda_k(g_t^{(m,n)}x_\btheta),\]
and $\lambda_k(\cdot)$ denotes the $k$-th successive minimum of a lattice in $\RR^{m+n}$.

It is easy to see that, up to a finite error, $\bh(t)$ is piecewise linear with few possible slopes. Minkowski's first and second convex body theorems give further information. In a landmark paper \cite{Roy1}, Roy showed that when $m$ or $n$ is $1$, the successive minima functions are precisely approximated by \emph{Roy-systems}, which are relatively simple combinatorial objects. In \cite{DFSU17,VarPrinc}, Das, Fishman, Simmons and Urba\'nski extended this result to arbitrary $m$ and $n$ and quantified the result.

\begin{definition}[\cite{VarPrinc}]\label{d-template}
Let $(m,n)\in \NN^2$ and $I\subset [0,\infty)$ be an interval.
An \emph{$m \times n$ template on $I$} is a piecewise linear continuous map $\bL=(L_1, \ldots, L_{m+n}) : I \to \RR^{m+n}$ satisfying the following conditions:
\begin{enumerate}
\item{$L_1 \le L_2 \le \cdots \le L_{m+n}$}.\\
\item{The derivative $L_j'(t)$, when well-defined, satisfies $-1/n\le L_j'(t) \le 1/m $.}\\
\item{For any $1\le j\le m+n$ and subinterval $J\subset I$ such that $L_j < L_{j+1}$ on $J$ (with the convention $L_{m+n+1} = +\infty$), the restriction on $J$ of the function
$F_j := \sum_{0< k \le j }L_k$ is convex with slopes in the set
\[Z(j) := \left\{\frac{k_1}{m} - \frac{k_2}{n} :  0\le k_1 \le m, 0\le k_2\le n, k_1+k_2=j\right\}.\]
}
\end{enumerate}
\end{definition}

Generalizing Roy's theorem \cite{Roy1}, it is shown in \cite{VarPrinc} that for every $\btheta \in M_{m\times n} (\RR)$, there is an $m\times n$ template $\bL$ on $[0,\infty)$ such that $\bh_\btheta-\bL$ is bounded, and conversely, for every such template $\bL$, there exists $\btheta \in M_{m\times n} (\RR)$ such that $\bh_\btheta - \bL$ is bounded.
The variational principle provides a quantitative version of the latter statement. It is expressed in terms of the \emph{lower average contraction rate} of a template, as described below.

For a template $\bL$ on $I\subset [0,\infty)$ and a subinterval $[T_1,T_2]\subset I$, one can define the \emph{average contraction rate} $\Delta(\bL, [T_1,T_2])$. As the definition is relatively long, we refer the reader to \cite[Definition 2.5]{VarPrinc}.
When $I=[0,\infty)$, we write $\Delta(\bL,T)=\Delta(\bL, [0,T])$, and define the \emph{lower average contraction rate} $\underline{\delta}(\bL)$ of $\bL$ as
\[\underline{\delta}(\bL) = \liminf_{T\to\infty} \Delta(\bL,T). \]

It is clear that the constant function $\bL=\bf{0}$ is a template, called the \emph{trivial} $m\times n$ template. Its lower average contraction rate is given below.

\begin{lemma}\label{r-trivial-dim}
The lower average contraction rate of the trivial $m\times n$ template on $[0,\infty)$ is $mn$.
\end{lemma}

\begin{proof}
This follows directly from \cite[Definition 2.5]{VarPrinc} (see also \cite[Section 28]{VarPrinc}).
\end{proof}

For an $m \times n$ template $\bL$ on $[0,\infty)$, set
\[\mathcal{M}(\bL) = \{ \btheta \in M_{m \times n}(\RR) : \bh_\btheta-\bL \textrm{ is bounded}\}.\]
More generally, given a collection $\mathcal{L}$ of $m \times n$ templates on $[0,\infty)$, denote
\[\mathcal{M}(\mathcal{L}) = \bigcup_{\bL\in\mathcal{L}} \mathcal{M}(\bL) . \]
The collection $\mathcal{L}$ is said to be \emph{closed under finite perturbations} if whenever $\bL$ and $\bL'$ are templates such that $\bL\in\mathcal{L}$ and $\bL-\bL'$ is bounded then $\bL'\in\mathcal{L}$.
The variational principle reads as follows.

\begin{theorem}[\cite{VarPrinc}]\label{t-dfsu2}
Let $\mathcal{L}$ be a Borel collection of $m \times n$ templates on $[0,\infty)$ that is closed under finite perturbations. Then
$$\dim \mathcal{M}(\mathcal{L})  = \sup_{\bL\in\mathcal{L}} \underline{\delta}(\bL).$$
\end{theorem}

\subsection{Reformulation of joint singularity properties and strategy of proof}

Let us fix an integer $s\ge 2$, a pair $(m_i, n_i)\in \NN^2$ for each $1\le i\le s$, and a weight vector $\ba=(a_1, \ldots, a_s)\in \RR_+^s$. Then we can reformulate various joint singularity properties using the successive minima function.

\begin{lemma}\label{p-smf}
Let $\bTheta=(\btheta_1, \ldots, \btheta_s)\in \bM$, $\delta\in (0,1]$.
\begin{enumerate}
  \item $\bTheta\in D(F^+_\ba, \bM)$ if and only if
\begin{equation}\label{e:smf1}
\limsup_{t\to\infty} \min_{1\le i \le s} h_{\btheta_i,1}(a_i t) =-\infty.
\end{equation}
  \item $\bTheta\in D^e(F^+_\ba, \bM)$ if and only if \eqref{e:smf1} holds and
\[\limsup_{t\to\infty}\min_{\substack{1\le i \le s\\ i\ne j}}h_{\btheta_i,1}(a_i t)> - \infty \quad \text{ for all } 1\le j \le s.\]
  \item $\bTheta\in D_{\delta}(F^+_\ba, \bM)$ if and only if
\begin{equation}\label{e:smf2}
\limsup_{T\to\infty} \frac{1}{T}\int_0^T \mathbbm{1}_{[-C, C]} \left( \min_{1\le i \le s} h_{\btheta_i,1}(a_i t)\right)  \dd t\le 1-\delta  \quad \textrm{ for all } C>0.
\end{equation}
  \item $\bTheta\in D^e_{\delta}(F^+_\ba, \bM)$ if and only if \eqref{e:smf2} holds and there exists $C>0$ such that
\[\limsup_{T\to \infty}\frac{1}{T}\int_0^T \mathbbm{1}_{[-C, C]} \left(\min_{\substack{1\le i \le s \\ i\ne j}} h_{\btheta_i,1}(a_i t)\right)\dd t> 1-\delta \quad \textrm{ for all } 1\le j \le s. \]
\end{enumerate}
 \end{lemma}

\begin{proof}
These are direct consequences of the definitions of the joint singularity properties and Mahler's compactness criterion.
\end{proof}

To prove Proposition \ref{p-lower-bound}, we are going to construct $s$-tuples of templates $(\bL^1, \ldots , \bL^s)$ such that $\prod_{i=1}^s \cM(\bL^i)$ is contained in  $D^e(F_{\ba}^+, \bM)$ or $D^e_{\delta}(F_{\ba}^+, \bM)$.
Then the desired lower bounds for the Hausdorff dimensions of $D^e(F_{\ba}^+, \bM)$ and $D^e_{\delta}(F_{\ba}^+, \bM)$  follow from Theorem \ref{t-dfsu2}. Here and in what follows, when saying that $(\bL^1, \ldots , \bL^s)$ is an $s$-tuple of templates, we always assume $\bL^i$ is an $m_i\times n_i$ template on $[0,\infty)$, and write the $j$-th component of $\bL^i$ as $L^i_j$. Let $b_i$ be the number defined in \eqref{e:bi}, and suppose that \eqref{e:b1} holds. We will construct tuples of templates satisfying the following two lemmas.

 \begin{lemma}\label{p-lower-key-1}
There exists an $s$-tuple of templates $(\bL^1, \ldots , \bL^s)$ satisfying
\begin{eqnarray}
\underline{\delta}(\bL^1) &=& m_1n_1-b_1, \label{e-ess-dim-1}\\
\underline{\delta}(\bL^i) &=& m_in_i \quad \textrm{ for all } 2\le i \le s, \label{e-ess-dim-2}\\
\limsup_{t\to\infty}\min_{1\le i \le s}L^i_1(a_i t)&=& -\infty, \label{e-ess-div-1}\\
\limsup_{t\to\infty}\min_{\substack{1\le i \le s \\ i\ne j}}L^i_1(a_i t)& =& 0 \quad \textrm{ for all } 1\le j \le s. \label{e-ess-div-2}
\end{eqnarray}
\end{lemma}

\begin{lemma}\label{p-lower-e-key}
 Let $\delta\in (0,1]$ and $\delta_1,\ldots,\delta_s \in (0,\delta)$ be such that $\sum_{i=1}^s\delta_i = \delta$. Then there exists an $s$-tuple of templates $(\bL^1, \ldots , \bL^s)$ such that
 \begin{equation}\label{e-del-dim}
 \underline{\delta}(\bL^i) = m_in_i-\delta_ib_i \quad  \text{ for all } 1\le i\le s,
 \end{equation}
  \begin{equation}\label{e-del-div-1}
\limsup_{T\to \infty}\frac{1}{T}\int_0^T \mathbbm{1}_{[-C, C]} \left(\min_{\substack{1\le i \le s \\ i\ne j}} L^i_1(a_it)\right)\dd t = 1-\sum_{i\ne j}\delta_i  \quad \text{ for all } 1\le j\le s, C>0,
 \end{equation}
  \begin{equation}\label{e-del-div-2}
\limsup_{T\to \infty}\frac{1}{T}\int_0^T \mathbbm{1}_{[-C, C]} \left(\min_{1\le i\le s}L^i_1(a_it)\right)\dd t = 1-\delta  \quad \text{ for all } C>0.
 \end{equation}
\end{lemma}

We postpone the proofs of Lemmas \ref{p-lower-key-1} and \ref{p-lower-e-key} and first deduce Proposition \ref{p-lower-bound} from them.

\begin{proof}[Proof of Proposition \ref{p-lower-bound}]
Let $(\bL^1, \ldots , \bL^s)$ be an $s$-tuple of templates satisfying \eqref{e-ess-dim-1}--\eqref{e-ess-div-2}. By \eqref{e-ess-div-1}, \eqref{e-ess-div-2} and Lemma \ref{p-smf}, we have
\[\prod_{i=1}^s \cM(\bL^i) \subset D^e(F_{\ba}^+, \bM). \]
On the other hand, if we let $\mathcal{L}^i$ denote the collection of templates $\bL$ such that $\bL^i-\bL$ is bounded, then $\mathcal{L}^i$ is Borel and is closed under finite perturbations, and hence it follows from Theorem \ref{t-dfsu2} that
\[\dim \cM(\bL^i)=\dim \cM(\mathcal{L}^i)=\sup_{\bL\in\mathcal{L}^i} \underline{\delta}(\bL)\ge\underline{\delta}(\bL^i).\]
These properties, together with \eqref{e-ess-dim-1} and \eqref{e-ess-dim-2}, imply that
\begin{align*}
 \dim D^e(F^+_\ba, \bM) &\ge \dim \prod_{i=1}^s \cM(\bL^i) \ge \sum_{i=1}^s  \dim \cM(\bL^i) \\
 &\ge \sum_{i=1}^s  \underline{\delta}(\bL^i)=\sum_{i=1}^s  m_in_i - b_1.
 \end{align*}
This proves \eqref{e-lower-bound-1}.

The proof of \eqref{e-lower-bound-2} is similar. Let $\delta_1,\ldots,\delta_s \in (0,\delta)$ be such that $\sum_{i=1}^s\delta_i = \delta$, and let $(\bL^1, \ldots , \bL^s)$ be an $s$-tuple of templates satisfying \eqref{e-del-dim}--\eqref{e-del-div-2}. We claim that
\[\prod_{i=1}^s \cM(\bL^i) \subset D^e_{\delta}(F_{\ba}^+, \bM). \]
In fact, if $\bTheta=(\btheta_1, \ldots, \btheta_s)\in\prod_{i=1}^s \cM(\bL^i)$, and if we denote
\[C_0=\sup_{t\ge0,1\le i\le s}|h_{\btheta_i,1}(t)-L^i_1(t)|,\]
then by \eqref{e-del-div-1}, for $1\le j\le s$, we have
\begin{align*}
& \limsup_{T\to \infty}\frac{1}{T}\int_0^T \mathbbm{1}_{[-C_0-1, C_0+1]} \left( \min_{\substack{1\le i \le s \\ i\ne j}} h_{\btheta_i,1}(a_i t)\right)  \dd t \\
\ge \ & \limsup_{T\to \infty}\frac{1}{T}\int_0^T \mathbbm{1}_{[-1, 1]} \left( \min_{\substack{1\le i \le s \\ i\ne j}} L^i_1(a_i t)\right)  \dd t
=  1-\sum_{i\ne j}\delta_i > 1-\delta,
\end{align*}
and by \eqref{e-del-div-2}, for any $C>0$, we have
\begin{align*}
& \limsup_{T\to\infty} \frac{1}{T}\int_0^T \mathbbm{1}_{[-C, C]} \left( \min_{1\le i \le s} h_{\btheta_i,1}(a_i t)\right)  \dd t \\
\le \ & \limsup_{T\to \infty}\frac{1}{T}\int_0^T \mathbbm{1}_{[-C-C_0, C+C_0]} \left(\min_{1\le i\le s}L^i_1(a_it)\right)\dd t = 1-\delta.
\end{align*}
It then follows from Lemma \ref{p-smf} that $\bTheta\in D^e_{\delta}(F_{\ba}^+, \bM)$. This verifies the claim.
Similar to the above case, Theorem \ref{t-dfsu2} implies that $\dim \cM(\bL^i)\ge\underline{\delta}(\bL^i)$.
Then by \eqref{e-del-dim}, we have
\begin{align*}
 \dim D_{\delta}^e(F^+_\ba, \bM) &\ge \dim \prod_{i=1}^s \cM(\bL^i) \ge\sum_{i=1}^s  \dim \cM(\bL^i) \\
 &\ge \sum_{i=1}^s  \underline{\delta}(\bL^i)= \sum_{i=1}^s  m_in_i - \sum_{i=1}^s  \delta_ib_i.
 \end{align*}
By letting $\delta_1=\delta-(s-1)\epsilon$ and $\delta_2=\cdots=\delta_s=\epsilon$ for sufficiently small $\epsilon>0$ and taking $\epsilon\to0$, we obtain
 $$\dim D_{\delta}^e(F^+_\ba, \bM) \ge \sum_{i=1}^s  m_in_i-\delta b_1.$$
 This completes the proof.
\end{proof}

\subsection{Standard templates}\label{construction}
The rest of this section is devoted to the proofs of Lemmas \ref{p-lower-key-1} and \ref{p-lower-e-key}. Our proofs will be constructive.
In this subsection, we first recall the notion of \emph{standard template} defined by two points introduced in \cite{VarPrinc}, which will be the building blocks of our construction.

\begin{definition}[\cite{VarPrinc}]
Given two points $(t', \eps'),(t'', \eps'')\in[0,\infty)^2$ with $t'<t''$, and denote $\Delta t = t''-t'$, $\Delta \eps = \eps''-\eps'$. Let us say that the pair of points $((t', \eps'), (t'', \eps''))$ is \emph{admissible} if it satisfies the following conditions:
\begin{equation}\label{condST1}
-\frac{\Delta t}{m} \le \Delta \eps \le \frac{\Delta t}{n},
\end{equation}
\begin{equation}\label{condST2}
\Delta \eps \ge - \frac{n-1}{2n}\Delta t \; \textrm{ if } m=1, \textrm{ and } \Delta\eps \le \frac{m-1}{2m} \Delta t \; \textrm{ if } n=1,
\end{equation}
\begin{equation}\label{condST3}
(n-1)\Big(\frac{\Delta t}{n}  - \Delta \eps\Big) \ge (m+n) \eps'  \textrm{ \ or \ } (m-1) \Big(\frac{\Delta t}{m}  + \Delta \eps\Big ) \ge (m+n) \eps''.
\end{equation}
The \emph{standard template} $\bL((t', \eps'), (t'', \eps''))$ associated to an admissible pair $((t', \eps'), (t'', \eps''))$ is the $m\times n$ template $(L_1, \ldots, L_{m+n})$ on $[t',t'']$ defined in the following way.
\begin{itemize}
  \item Let $g_1, g_2 : [t', t''] \to \RR$ be piecewise linear functions such that
  $$g_1(t')=g_2(t') = -\eps', \quad g_1(t'')=g_2(t'')=-\eps'',$$
   and $g_i$ has two intervals of linearity: one on which $g'_i = 1/m$ and the other  on which $g'_i = - 1/n$ . For $i = 1$ the latter interval comes first while for $i = 2$ the former interval comes first. The existence of such functions $g_1$ and $g_2$ is guaranteed by \eqref{condST1}. Finally, let $g_3 = \cdots = g_{m+n}$ be functions on  $[t', t'']$ chosen so that $g_1(t) + \cdots + g_{m+n}(t) = 0$ for all $t\in [t', t'']$.
  \item Let $t \in [t', t'']$. If $g_2(t) \leq g_3(t)$, define $L_j(t)=g_j(t)$ for all $1\le j\le m+n$.
   Otherwise, define $L_1(t) = g_1(t)$, and define $L_2(t) = \cdots= L_{m+n}(t)$ so that $L_1(t)+\cdots+L_{m+n}(t) =0$.
\end{itemize}

Moreover, let us say that a finite sequence of points $\{(t_l, \eps_l)\}_{1\le l \le k}$ is \emph{admissible} if for all $1\le l\le k-1$, the pair $((t_l, \eps_l), (t_{l+1}, \eps_{l+1}))$ is admissible. We define the \emph{standard template} associated to $\{(t_l, \eps_l)\}_{1\le l \le k}$ to be the template on the interval $[t_1, t_k]$ that equals $\bL((t_l, \eps_l), (t_{l+1}, \eps_{l+1}))$ on $[t_l, t_{l+1}]$.
\end{definition}

We will need the following statement.

\begin{lemma}\label{deltaST}
  Let $\bL$ be the standard template associated to an admissible pair of points $(t', \eps')$ and $(t'', \eps'')$. Then
  \begin{itemize}
    \item[(1)] $L_1(t)\le -\min \{\eps', \eps''\}$ for all $t\in [t', t'']$.
    \item[(2)] The average contraction rate on $[t',t'']$ is given by
\[ \Delta(\bL,[t',t'']) = mn-\frac{mn}{m+n} - O\left(\frac{\max(\eps', \eps'')}{t''-t'}\right).\]
  \end{itemize}
\end{lemma}

\begin{proof}
 (1) follows directly from the definition. For the proof of (2), see the paragraph below Definition 12.4 in \cite{VarPrinc}.
\end{proof}

The following simple observation will be also useful.

\begin{lemma}\label{l-adm}
Any pair of points $((t', \eps'), (t'', \eps''))$ that satisfies
\begin{equation}\label{e-admiss}
  t''-t'\ge (m+n)^2 \max(\eps', \eps'')
\end{equation}
is admissible.
\end{lemma}
\begin{proof}
Given \eqref{e-admiss}, the conditions \eqref{condST1}, \eqref{condST2} and \eqref{condST3} are easily checked.
\end{proof}

\subsection{Construction of templates (I)  }
Now let us begin to prove Lemma \ref{p-lower-key-1}. In this subsection, we construct the templates we need and in the next one, we verify that they satisfy the conditions of Lemma \ref{p-lower-key-1}.

Set $T_0=1$ and $T_{k+1}=T_k+\sqrt{T_k}$.   Set
$$l_k=\big[\sqrt[3]{T_k}\,\big],\quad   \gamma_k=l_k^{-1}\sqrt{T_k},$$
and for $0\le l\le l_k$, set $t_{k,l}=T_k + l\gamma_k$. Clearly, we have $t_{k,0}=T_k$ and $t_{k, l_k}=T_{k+1}$.  Note that $\gamma_k$ goes to infinity when $k$ goes to infinity. So, there exists $k_0>0$ such that for any $k\ge k_0$,
\begin{equation}\label{e-valid}
  l_k \ge 8s, \quad \text{ and } \quad a_i\gamma_k\ge (m_i+n_i)^2 \log \gamma_{k} \text{ for all } 1\le i\le s .
\end{equation}

We define  $\bL^1$ as follows (see Figure \ref{fig1} for the $s=2$ case):
\begin{itemize}
  \item On the interval $[0, a_1T_{k_0}]$, set  $\bL^1$ to be the trivial template.
  \item On the interval $[a_1T_{k_0}, a_1T_{k_0+1}]$, set $\bL^1$ to be the standard template associated to the pair of points
      \[(a_1T_{k_0}, 0), (a_1T_{k_0+1}, \log\gamma_{k_0+1}).\]
  \item Let $k\ge k_0+1$. On the subinterval $[a_1t_{k, 4s-4}, a_1T_{k+1}]$ of $[a_1 T_k, a_1 T_{k+1}]$, set $\bL^1$ to be the standard template associated to the sequence of points
      \[(a_1t_{k, 4s-4},\log \gamma_k), (a_1t_{k, 4s-3},\log \gamma_k), \ldots, (a_1t_{k, l_k-1},\log \gamma_k), (a_1 T_{k+1}, \log \gamma_{k+1}).\]
      For $0\le l\le s-2$, on the subinterval $[a_1t_{k, 4l}, a_1t_{k, 4l+4}]$, set $\bL^1$ to be the standard template associated to the sequence of points
      \[(a_1 t_{k, 4l},\log \gamma_k), (a_1 t_{k, 4l+1},\log \gamma_k), (a_1 t_{k, 4l+2},0), (a_1 t_{k, 4l+3},\log \gamma_k), (a_1 t_{k,4l+4}, \log \gamma_k).\]
\end{itemize}
According to Lemma \ref{l-adm} and \eqref{e-valid}, all the sequences of points appearing above are admissible, hence the construction is valid.

Also, we define $\bL^i$ for $2\le i\le s$ as follows (see Figure \ref{fig1} for the $s=2$ case):
\begin{itemize}
  \item On the interval $[0, a_iT_{k_0}]$, set  $\bL^i$ to be the trivial template.
  \item On the interval $[a_iT_{k_0}, a_iT_{k_0+1}]$, set $\bL^i$ to be the standard template associated to the pair of points
      \[(a_iT_{k_0}, 0), (a_iT_{k_0+1}, \log\gamma_{k_0+1}).\]
  \item Let $k\ge k_0+1$. On the subinterval $[a_iT_k, a_it_{k, 4i-8}]$ of $[a_i T_k, a_i T_{k+1}]$, set $\bL^i$ to be the trivial template. On the subinterval $[a_i t_{k, 4i-8}, a_i t_{k, 4i-4}]$, set $\bL^i$ to be the standard template associated to the sequence of points
      \[(a_i t_{k, 4i-8},0), (a_i t_{k, 4i-7},\log \gamma_k), (a_i t_{k, 4i-6},\log \gamma_k),(a_i t_{k, 4i-5},\log \gamma_k), (a_i t_{k, 4i-4},0).\]
      On the subinterval $[a_i t_{k, 4i-4}, a_i T_{k+1}]$, set $\bL^i$ to be the trivial template.
 \end{itemize}
  According to Lemma \ref{l-adm} and \eqref{e-valid}, all the sequences of points appearing above are admissible, hence the construction is valid.
\\

\begin{figure}[h!]
 \begin{center}
 \begin{tikzpicture}[scale=0.8]
 \draw[black, semithick] (15,0)--(20,0) node [right,black] { $t$ } ;
 \draw[black, semithick] (0,0)--(13,0) node [right,black] {  } ;

 \draw (14,0) node{$\cdots$};

  \fill (1,0) circle[radius=2pt] node [above,black] { $T_k$ } ;
    \draw[black, dashed] (1,0)--(1,-8) ;
  \fill (19,0) circle[radius=2pt] node [above,black] { $T_{k+1}$ } ;
   \draw[black, dashed] (19,0)--(19,-8) ;
  \fill (3,0) circle[radius=2pt] node [above,black] { $t_{k,1}$ } ;
    \draw[black, dashed] (3,0)--(3,-8) ;
  \fill (5,0) circle[radius=2pt] node [above,black] { $t_{k,2}$ } ;
    \draw[black, dashed] (5,0)--(5,-8) ;
    \fill (7,0) circle[radius=2pt] node [above,black] { $t_{k,3}$ } ;
    \draw[black, dashed] (7,0)--(7,-8) ;
    \fill (9,0) circle[radius=2pt] node [above,black] { $t_{k,4}$ } ;
    \draw[black, dashed] (9,0)--(9,-8) ;
    \fill (11,0) circle[radius=2pt] node [above,black] { $t_{k,5}$ } ;
    \draw[black, dashed] (11,0)--(11,-8) ;

  \fill (17,0) circle[radius=2pt] node [above,black] { $t_{k,l_k-1}$ } ;
      \draw[black, dashed] (17,0)--(17,-8) ;

    %L_2

      \draw[black, dashed] (15,-3)--(20,-3) node [right,black] { $\bL^1(a_1 t)$ } ;
       \draw[black, dashed] (0,-3)--(13,-3) ;
       \draw (14,-3) node{$\cdots$};

    % \draw[black, semithick] (1,-2.8) --(1.21,-2.9);
    % \draw[black, semithick] (1,-3.4) --(1.21,-2.9);
    % \draw[black, semithick] (1.21,-2.9)-- (1.6, -2.8);
    % \draw[black, semithick] (1.6, -2.8) --(2.6,-2.3);
    % \draw[black, semithick] (1.6, -2.8) --(3,-3.6);
    % \draw[black, semithick] (2.6,-2.3) --(3,-2.6);
    % \draw[black, semithick] (1,-3.4) --(2.6,-4.4);
    % \draw[black, semithick] (2.6,-4.4) --(3,-3.6);

       \draw[black, semithick] (1,-2.6)--(1.5,-2.8);
      \draw[black, semithick] (1,-3.6)--(1.5,-2.8);
      \draw[black, semithick] (1,-3.6)--(2.4,-4.6);
      \draw[black, semithick] (2.4,-4.6)--(3.5,-2.8);
      \draw[black, semithick] (1.5,-2.8)--(1.7,-2.7);
      \draw[black, semithick] (1.7,-2.7)--(2.4,-2.3);
      \draw[black, semithick] (1.7,-2.7)--(4.2,-4.4);
      \draw[black, semithick] (2.4,-2.3)--(3.5,-2.8);
      \draw[black, semithick] (3.5,-2.8)--(5,-3);
      \draw[black, semithick] (4.2,-4.4)--(5,-3);

    % \draw[black, semithick] (3,-3.6)--(4.4,-4.2);
    % \draw[black, semithick] (3,-3.6)--(3.4,-2.8);
    % \draw[black, semithick] (3,-2.6)--(3.4,-2.8);
    % \draw[black, semithick] (4.4,-4.2)--(5,-3);
    % \draw[black, semithick] (3.4,-2.8)--(5,-3);
    % \draw[black, semithick] (4,-2.6)--(5,-3);

      \draw[black, semithick] (5,-3)--(6,-2.8);
      \draw[black, semithick] (5,-3)--(6.6,-4.2);
      \draw[black, semithick] (6,-2.8)--(6.6,-2.4);
      \draw[black, semithick] (6,-2.8)--(7,-3.6);
      \draw[black, semithick] (6.6,-4.2)--(7,-3.6);
      \draw[black, semithick] (6.6,-2.4)--(7,-2.6);

      \draw[black, semithick] (7,-2.6)--(7.5,-2.9);
      \draw[black, semithick] (7,-3.6)--(7.5,-2.9);
      \draw[black, semithick] (7,-3.6)--(8.4,-4.4);
      \draw[black, semithick] (8.4,-4.4)--(9,-3.6);
      \draw[black, semithick] (7.5,-2.9)--(7.8,-2.8);
      \draw[black, semithick] (7.8,-2.8)--(9,-3.6);
      \draw[black, semithick] (8.4,-2.2)--(9,-2.6);
      \draw[black, semithick] (7.8,-2.8)--(8.4,-2.2);

      \draw[black, semithick] (9,-2.6)--(9.5,-2.9);
      \draw[black, semithick] (9,-3.6)--(9.5,-2.9);
      \draw[black, semithick] (9,-3.6)--(10.4,-4.4);
      \draw[black, semithick] (10.4,-4.4)--(11,-3.6);
      \draw[black, semithick] (9.5,-2.9)--(9.8,-2.8);
      \draw[black, semithick] (9.8,-2.8)--(11,-3.6);
      \draw[black, semithick] (10.4,-2.2)--(11,-2.6);
      \draw[black, semithick] (9.8,-2.8)--(10.4,-2.2);

      \draw[black, semithick] (11,-2.6)--(11.5,-2.9);
      \draw[black, semithick] (11,-3.6)--(11.5,-2.9);
      \draw[black, semithick] (11,-3.6)--(12.4,-4.4);
      \draw[black, semithick] (12.4,-4.4)--(13,-3.6);
      \draw[black, semithick] (11.5,-2.9)--(11.8,-2.8);
      \draw[black, semithick] (11.8,-2.8)--(13,-3.6);
      \draw[black, semithick] (12.4,-2.2)--(13,-2.6);
      \draw[black, semithick] (11.8,-2.8)--(12.4,-2.2);

      \draw[black, semithick] (17,-2.6)--(17.5,-2.9);
      \draw[black, semithick] (17,-3.6)--(17.5,-2.9);
      \draw[black, semithick] (17,-3.6)--(18.4,-4.4);
      \draw[black, semithick] (18.4,-4.4)--(19,-3.6);
      \draw[black, semithick] (17.5,-2.9)--(17.8,-2.8);
      \draw[black, semithick] (17.8,-2.8)--(19,-3.7);
      \draw[black, semithick] (18.4,-2.2)--(19,-2.5);
      \draw[black, semithick] (17.8,-2.8)--(18.4,-2.2);

      %L_1
      \draw[black, dashed] (0,-7)--(9,-7) ;
      \draw[black, dashed] (15,-7)--(20,-7)node [right,black] { $\bL^2(a_2 t)$ }  ;
      \draw[black, semithick] (9,-7)--(13,-7);
      \draw[black, semithick] (15,-7)--(20,-7);

      \draw (14,-7) node{$\cdots$};

   %   \draw[black, semithick] (8,-7.6)--(10.8,-9.2);
   %   \draw[black, semithick] (10.8,-9.2)--(12,-7.6);
   %   \draw[black, semithick] (8,-7.6)--(8.5,-6.9);
   %   \draw[black, semithick] (8,-6.6)--(8.5,-6.9);
   %   \draw[black, semithick] (8.5,-6.9)--(8.8,-6.8);
   %   \draw[black, semithick] (8.8,-6.8)--(10.2,-6.4);
   %   \draw[black, semithick] (10.2,-6.4)--(12,-7.6);
   %   \draw[black, semithick] (10.2,-6.4)--(10.8,-5.8);
   %   \draw[black, semithick] (10.8,-5.8)--(12,-6.6);

       \draw[black, semithick] (1,-7)--(2,-6.8);
      \draw[black, semithick] (1,-7)--(2.6,-8.2);
      \draw[black, semithick] (2,-6.8)--(2.6,-6.4);
      \draw[black, semithick] (2,-6.8)--(3,-7.6);
      \draw[black, semithick] (2.6,-8.2)--(3,-7.6);
      \draw[black, semithick] (2.6,-6.4)--(3,-6.6);

      \draw[black, semithick] (3,-6.6)--(3.5,-6.9);
      \draw[black, semithick] (3,-7.6)--(3.5,-6.9);
      \draw[black, semithick] (3,-7.6)--(4.4,-8.4);
      \draw[black, semithick] (4.4,-8.4)--(5,-7.6);
      \draw[black, semithick] (3.5,-6.9)--(3.8,-6.8);
      \draw[black, semithick] (3.8,-6.8)--(5,-7.6);
      \draw[black, semithick] (4.4,-6.2)--(5,-6.6);
      \draw[black, semithick] (3.8,-6.8)--(4.4,-6.2);

      \draw[black, semithick] (5,-6.6)--(5.5,-6.8);
      \draw[black, semithick] (5,-7.6)--(5.5,-6.8);
      \draw[black, semithick] (5,-7.6)--(6.4,-8.6);
      \draw[black, semithick] (6.4,-8.6)--(7.5,-6.8);
      \draw[black, semithick] (5.5,-6.8)--(5.7,-6.7);
      \draw[black, semithick] (5.7,-6.7)--(6.4,-6.3);
      \draw[black, semithick] (5.7,-6.7)--(8.2,-8.4);
      \draw[black, semithick] (6.4,-6.3)--(7.5,-6.8);
      \draw[black, semithick] (7.5,-6.8)--(9,-7);
      \draw[black, semithick] (8.2,-8.4)--(9,-7);

  \end{tikzpicture}
 \end{center}
 \caption{The $s=2$ case of $\bL^1(a_1 t)$ and $\bL^2(a_2 t)$ on the interval $[T_k,T_{k+1}]$, $k\ge k_0+1$. }\label{fig1}
 \end{figure}
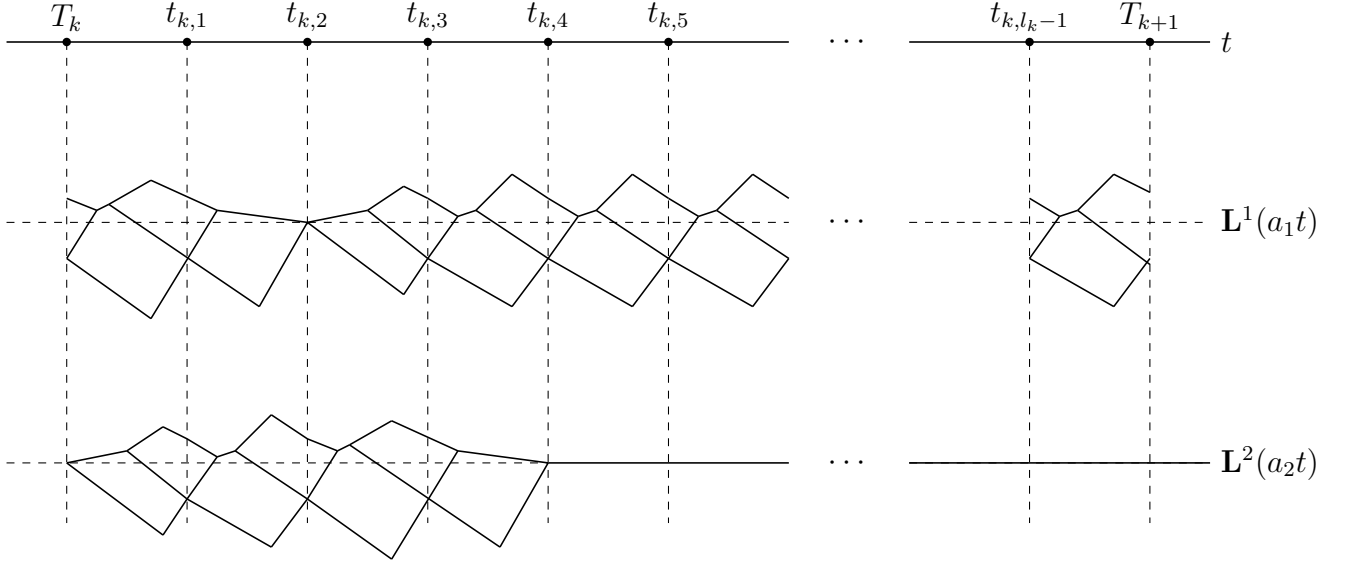

\subsection{Proof of Lemma \ref{p-lower-key-1}}\label{s-4}

By Lemma \ref{deltaST}(1) and our construction, for $k\ge k_0+1$, we have
\[L^1_1(a_1t)\le -\log \gamma_k \text{ on the subset } [T_k, T_{k+1}]\setminus \bigcup_{0\le l\le s-2} [t_{k, 4l+1,}, t_{k, 4l+3}], \]
\[L^i_1(a_it)\le -\log \gamma_k \text{ on the interval } [t_{k, 4i-7}, t_{k, 4i-5}] \text{ for } 2\le i\le s.  \]
Thus,
\[\lim_{k\to \infty}\max_{T_k\le t\le T_{k+1}}\min_{1\le i\le s}  L^i_1(a_it)\le \lim_{k\to \infty}-\log \gamma_k
= -\infty, \]
 which proves \eqref{e-ess-div-1}.

 It is clear from the definition that, for any $m\times n$ template $\bL=(L_1, \ldots, L_{m+n})$ defined on $[0, \infty)$, we always have $L_1(t)\le 0$ for all $t\ge 0$. So the only non-obvious part of \eqref{e-ess-div-2}  is the $``\ge"$ part, the proof of which will be divided into two cases. When $j=1$, it is clear from the construction that, for $k\ge k_0+1$, we have
\[L^i_1(a_i t_{k, 4s})=0 \text{ for all } 2\le i \le s.\]
Thus,
\[\lim_{k\to \infty}\max_{T_k\le t\le T_{k+1}}\min_{2\le i\le s}  L^i_1(a_it)\ge 0. \]
When $2\le j\le s$, it is also clear from the construction that, for $k\ge k_0+1$, we have
\[L^1_1(a_1 t_{k, 4j-6})=0 \text{ and } L^i_1(a_i t_{k, 4j-6})=0  \text{ for all } 2\le i \le s, i\ne j.\]
Thus,
\[\lim_{k\to \infty}\max_{T_k\le t\le T_{k+1}}\min_{\substack{1\le i\le s\\ i\ne j}}  L^i_1(a_it)\ge 0. \]
This completes the proof of  \eqref{e-ess-div-2}.

%On the other hand, by Lemma \ref{r-trivial-dim}, for $i\ge 3$, $\underline{\delta}(\bL_i) = m_in_i$.

Now we are going to prove \eqref{e-ess-dim-1} and \eqref{e-ess-dim-2}.  Note that for $1\le i\le s$, $k>0$,  and $T \in [a_iT_k, a_i T_{k+1}]$, we have
\[\Delta(\bL^i, T) - \Delta(\bL^i, a_iT_k) = O \left(\frac{T- a_iT_k}{a_iT_k}\right) = O \left( \frac{1}{\sqrt{T_k}}\right), \]
which goes to $0$ as $k$ tends to infinity.
Hence it suffices to compute $\Delta(\bL^i,T)$ at $a_iT_k$.
By definition,
$$\Delta\left(\bL^i, a_iT_k\right) = \sum_{0\le j\le k-1} \frac{T_{j+1}-T_j}{T_k}\Delta\left(\bL^i, [a_iT_j, a_iT_{j+1}]\right).$$
As $T_{j+1}-T_j=\sqrt{T_j}$ goes to infinity when $j$ goes to infinity, to complete  the proof, it suffices to show that
\begin{eqnarray*}
  \lim_{k\rightarrow \infty} \Delta\left(\bL^1, [a_1T_k, a_1T_{k+1}]\right) &=& m_1n_1-b_1, \\
  \lim_{k\rightarrow \infty} \Delta\left(\bL^i, [a_iT_k, a_iT_{k+1}]\right) &=& m_in_i.
\end{eqnarray*}
In view of Lemma \ref{deltaST}(2), we get
\begin{align*}
\lim_{k\rightarrow \infty}\Delta\left(\bL^1, [a_1T_k, a_1 T_{k+1}]\right) &= \lim_{k\rightarrow \infty}\sum_{0\le l\le l_k-1} \frac{t_{k,l+1}-t_{k,l}}{\sqrt{T_k}} \Delta\left(\bL^1, [a_1t_{k,l}, a_1t_{k,l+1}]\right)\\
&= m_1n_1-b_1+\lim_{k\rightarrow \infty}O\left(\frac{\log \gamma_k}{\gamma_k}\right)\\
&= m_1n_1-b_1
\end{align*}
and
\begin{align*}
&\ \lim_{k\rightarrow \infty}\Delta(\bL^i, [a_iT_k, a_iT_{k+1}])\\ &= \lim_{k\rightarrow \infty}\sum_{0\le l\le l_k-1} \frac{t_{k,l+1}-t_{k,l}}{\sqrt{T_k}} \Delta(\bL^i, [a_it_{k,l}, a_it_{k,l+1}])\\
&= \lim_{k\rightarrow \infty} \frac{1}{l_k}\left( (l_k-4s+4)m_in_i+(4s-4)\left(m_in_i-b_i+O\left(\frac{\log \gamma_k}{\gamma_k}\right)\right)\right)\\
&= m_in_i.
\end{align*}
Here we are using the fact that both $\gamma_k$ and $l_k$ tend to infinity when $k$ goes to infinity. This completes the proof of  Lemma \ref{p-lower-key-1}.
\qed

\begin{remark}\label{r-compare-sing}
Note that for the $\bL^1$  constructed above, any matrix $\btheta$ with $\bh_\btheta-\bL^1$ bounded  is $1$-singular, but not singular. In particular, it follows from Theorem \ref{t-dfsu2} and \eqref{e-ess-dim-1} that there are many $1$-singular matrices that are not singular.
\end{remark}

\subsection{Construction of templates (II)  }
We now prove Lemma \ref{p-lower-e-key}. As in the proof of Lemma \ref{p-lower-key-1}, we first construct the templates, and then verify the required properties.

Set $T_0=1$ and $T_{k+1}=T_k+\sqrt{T_k}$.   Set
$$l_k=\big[\sqrt[3]{T_k}\,\big], \quad \gamma_k=l_k^{-1}\sqrt{T_k},$$
and for $0\le l\le l_k$, set $t_{k,l}=T_k + l\gamma_k$. Clearly, we have $t_{k,0}=T_k$ and $t_{k, l_k}=T_{k+1}$.
Let us fix $\delta_i\in(0,\delta)$ with $\sum_{1\le i\le s} \delta_i=\delta$.
For any $k\ge 0$, set $q_k^0=0$ and for $1\le j\le s$, set
\[q_k^j= \max \left\{0\le l\le l_k: l\gamma_k \le \sum_{1\le i\le j}\delta_i\sqrt{T_k} \right\},\]
Clearly, there exists $k_0>0$ such that for any $k\ge k_0$, we have
\begin{equation}\label{e-valid-1}
  a_i\gamma_k\ge (m_i+n_i)^2 \log \gamma_{k} \text{ for all } 1\le i\le s,
\end{equation}
and
\begin{equation}\label{e-gap}
l_k-q_k^s \ge 4 \text{ and } q_k^{j+1}-q_k^j \ge 4 \text{ for all } 0\le j\le s-1.
\end{equation}
Note that by construction, for any $1\le i \le s$, we have
\begin{equation}\label{e-lim}
 \frac{q_k^{i}-q_k^{i-1}}{l_k} \to_{k\to\infty} \delta_i.
 \end{equation}

Now we define the templates $\bL^i$ as follows (see Figure \ref{fig2}):
\begin{itemize}
  \item On the interval $[0, a_iT_{k_0}]$, set  $\bL^i$ to be the trivial template.

   \item Let $k\ge k_0$. On the subinterval $[a_i T_k, a_it_{k, q_k^{i-1}}]$ of $[a_iT_k, a_iT_{k+1}]$, set $\bL^i$ to be the trivial template. On the subinterval $[a_it_{k, q_k^{i-1}}, a_it_{k,q_k^{i}}]$,  set $\bL^i$ to be the standard template associated to the sequence of points
            \[(a_i t_{k, q_k^{i-1}}, 0), (a_i t_{k,q_k^{i-1}+1},\log \gamma_k),\ldots,  (a_i t_{k,q_k^{i}-1}, \log \gamma_k), (a_i t_{k,q_k^{i}},0).\]
      On the subinterval $[a_i t_{k,q_k^{i}}, a_i T_{k+1}]$, set $\bL^i$ to be the trivial template.
      \end{itemize}
  According to Lemma \ref{l-adm} and \eqref{e-valid-1}, all the sequences of points appearing above are admissible. Hence the construction is valid.
\\

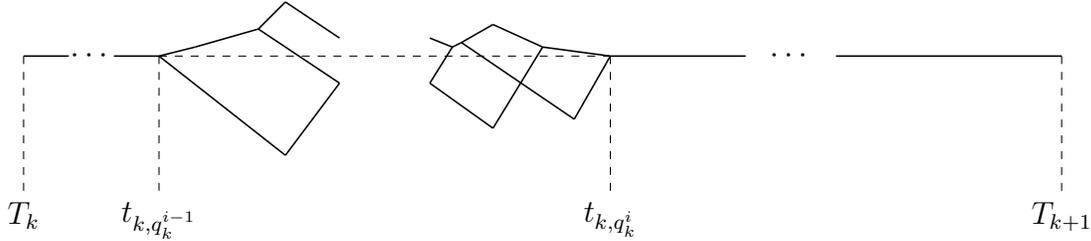
\begin{figure}[h!]
 \begin{center}
 \begin{tikzpicture}[scale=0.6]

%first block
   \draw[black, semithick] (5,1)--(7.8,-1.2);
      \draw[black, semithick] (7.8,-1.2)--(9,0.4);
      \draw[black, semithick] (5,1)--(5.8,1.2);
     \draw[black, semithick] (5.8,1.2)--(7.2,1.6);
     \draw[black, semithick] (7.2,1.6)--(9,0.4);
     \draw[black, semithick] (7.2,1.6)--(7.8,2.2);
    \draw[black, semithick] (7.8,2.2)--(9,1.4);

% dots 1

%block T_u,k
      \draw[black, semithick] (11,1.4)--(11.5,1.2);
      \draw[black, semithick] (11,0.4)--(11.5,1.2);
      \draw[black, semithick] (11,0.4)--(12.4,-0.6);
      \draw[black, semithick] (12.4,-0.6)--(13.5,1.2);
      \draw[black, semithick] (11.5,1.2)--(11.7,1.3);
      \draw[black, semithick] (11.7,1.3)--(12.4,1.7);
      \draw[black, semithick] (11.7,1.3)--(14.2,-0.4);
      \draw[black, semithick] (12.4,1.7)--(13.5,1.2);
    \draw[black, semithick] (13.5,1.2)--(15,1);
      \draw[black, semithick] (14.2,-0.4)--(15,1);

% 0

      \draw[black, semithick] (15,1)--(18,1);
      \draw[black, semithick] (20,1)--(25,1);
      \draw (19,1) node{$\cdots$};

      \draw[black, semithick] (2,1)--(3,1);
            \draw[black, semithick] (4,1)--(5,1);
      \draw (3.5,1) node{$\cdots$};

%axes & co
\draw[black, dashed, thin] (5,1)--(15,1) ;

\draw[black, dashed, thin] (2,1)--(2,-2) node[below] {$T_k$} ;
\draw[black, dashed, thin] (25,1)--(25,-2) node[below] {$T_{k+1}$} ;
\draw[black, dashed, thin] (15,1)--(15,-2) node[below] {$t_{k,q_k^i}$} ;
\draw[black, dashed, thin] (5,1)--(5,-2) node[below] {$t_{k,q_k^{i-1}}$} ;

   % \draw[<->] (8,-1)--(12,-1) node [midway,fill=white] { \Tiny{$h(t_{k,l+1})$} };

  \end{tikzpicture}
 \end{center}
 \caption{$\bL^i(a_i t)$ on the interval $[T_k,T_{k+1}]$, $k\ge k_0$. }\label{fig2}
 \end{figure}

 \subsection{Proof of Lemma \ref{p-lower-e-key} }
 Arguing as in Section \ref{s-4}, to prove \eqref{e-del-dim}, \eqref{e-del-div-1} and \eqref{e-del-div-2}, it suffices to show that, for all $1\le j\le s$,
 \begin{eqnarray}
   \lim_{k\rightarrow \infty}\Delta(\bL^j, [a_jT_k, a_jT_{k+1}]) &=& m_jn_j-\delta_j b_j\label{e-local-1}, \\
   \lim_{k\rightarrow \infty} \frac{1}{\sqrt{T_k}}\int_{T_k}^{T_{k+1}} \mathbbm{1}_{[-C, C]} \left(\min_{\substack{1\le i\le s \\ i\ne j}}L^i_1(a_i t)\right)\dd t &=& 1-\sum_{i\ne j}\delta_i\label{e-local-2}, \\
   \lim_{k\rightarrow \infty} \frac{1}{\sqrt{T_k}}\int_{T_k}^{T_{k+1}} \mathbbm{1}_{[-C, C]} \left(\min_{1\le i\le s}L^i_1(a_i t)\right)\dd t &=& 1-\delta.\label{e-local-3}
 \end{eqnarray}

According to our construction and Lemma \ref{deltaST}(2), we have
\begin{align*}
& \lim_{k\rightarrow \infty}\Delta(\bL^j, [a_jT_k, a_jT_{k+1}])\\
= \ & \lim_{k\rightarrow \infty}\sum_{0\le l\le l_k-1} \frac{t_{k,l+1}-t_{k,l}}{\sqrt{T_k}} \Delta(\bL^j, [a_jt_{k,l}, a_jt_{k,l+1}])\\
= \ & \lim_{k\rightarrow \infty} \frac{1}{l_k}\left( (l_k-q_k^{j}+q_k^{j-1})m_jn_j+(q_k^{j}-q_k^{j-1})\left(m_jn_j-b_j+O\left(\frac{\log \gamma_k}{\gamma_k}\right)\right)\right)\\
= \ & m_jn_j-\delta_j b_j.
\end{align*}
 This proves \eqref{e-local-1}.

According to our construction and Lemma \ref{deltaST}(1), for any $C>0$, we have
\begin{align*}
  &\ \lim_{k\rightarrow \infty} \frac{1}{\sqrt{T_k}}\int_{T_k}^{T_{k+1}} \mathbbm{1}_{[-C, C]} \left(\min_{\substack{1\le i\le s \\ i\ne j}}L^i_1(a_i t)\right)\dd t \\
  &=\lim_{k\rightarrow \infty} \sum_{\substack{1\le i\le s \\ i\ne j}}\frac{q_k^{i}-q_k^{i-1}}{l_k}+o(1) \\
   &= 1-\sum_{i\ne j}\delta_i.
\end{align*}
This proves \eqref{e-local-2}. The proof of \eqref{e-local-3} is similar, hence omitted.
This completes the proof of Lemma \ref{p-lower-e-key}.\qed

\end{document}